\documentclass[12pt]{article}

\usepackage{amsmath,amssymb,latexsym, amsbsy, amsfonts, amscd, amsthm, mathrsfs, xy,comment, xcolor}
\usepackage{multirow}
\usepackage{hyperref}
\usepackage{enumitem, empheq, url}
 \usepackage[bibtex-style]{amsrefs}
\xyoption{all}
\usepackage{tikz-cd}
\usepackage{rotating}

\theoremstyle{plain}

\newtheorem{theorem}{Theorem}[section]
\newtheorem{conjecture}[theorem]{Conjecture}
\newtheorem{prop}[theorem]{Proposition}
\newtheorem{corollary}[theorem]{Corollary}

\newtheorem{lemma}[theorem]{Lemma}

\theoremstyle{definition}
\newtheorem{definition}[theorem]{Definition}

\long\def\symbolfootnote[#1]#2{\begingroup
\def\thefootnote{\fnsymbol{footnote}}\footnote[#1]{#2}\endgroup}

\def\lra{\longrightarrow}

\DeclareMathOperator{\GL}{GL}

\def\fM{\mathfrak{M}}

\def\N{\mathrm{N}}
\def\1{\mf{1}}

\DeclareMathOperator{\Frac}{Frac}

\DeclareMathOperator{\res}{res}

\DeclareMathOperator{\Stab}{Stab}

\DeclareMathOperator{\ord}{ord}

\DeclareMathOperator{\rank}{rank}

\DeclareMathOperator{\trd}{trd}
\DeclareMathOperator{\height}{ht}

\newcommand{\mat}[4]{\begin{pmatrix}{#1} & {#2} \\ {#3} & {#4}
\end{pmatrix}}

\newcommand{\mf}{\mathfrak }

\newcommand{\mscr}{\mathscr}

\def\fa{\mathfrak{a}}

\def\fp{\mathfrak{p}}
\def\fq{\mathfrak{q}}

\def\fm{\mathfrak{m}}

\def\fl{\mathfrak{l}}
\def\fP{\mathfrak{P}}

\def\fm{\mathfrak{m}}

\def\Z{\mathbf{Z}}

\def\Q{\mathbf{Q}}
\def\C{\mathbf{C}}

\def\R{\mathbf{R}}

\def\bdf{\begin{defn}}
\def\edf{\end{defn}}

\def\cO{\mathcal{O}}

\def\fb{\mathfrak{b}}

\def\Gal{{\rm Gal}}

\def\sL{{\mscr L}}

\begin{document}
\baselineskip 15.8pt

\title{Ranks of matrices of logarithms of algebraic numbers I: the theorems of Baker and Waldschmidt--Masser}
\author{Samit Dasgupta}

\maketitle

\begin{abstract}
Let $\sL$ denote the $\Q$-vector space of logarithms of algebraic numbers.
 In this expository work, we provide an introduction to the study of ranks of matrices with coefficients in $\sL$. 
  We begin by considering a slightly different question, namely we present a proof of a weak form of Baker's Theorem.  This states that a collection of elements of $\sL$ that is linearly independent over $\Q$ is in fact linear independent over $\overline{\Q}$.  
  Next we recall Schanuel's Conjecture and prove Ax's analogue of it over $\C((t))$.
  
  We then consider arbitrary matrices with coefficients in $\sL$ and state the ``Structural Rank Conjecture,'' which gives a conjecture for the rank of a general matrix with coefficients in $\sL$.
  We prove the theorem of Waldschmidt and Masser, which provides a  lower bound giving a partial result toward the Structural Rank Conjecture.
We conclude by stating a new conjecture that we call the Matrix Coefficient Conjecture, which gives a necessary condition for a square matrix with coefficients in $\sL$ to be singular.
\end{abstract}

\tableofcontents

\section{Introduction}

At the 1900 International Congress of Mathematicians, David Hilbert presented 23 open problems that have continued to serve as an inspiration for generations of mathematicians.  As the 7th in his list, Hilbert asked the following question:

\bigskip

\noindent
{\bf Hilbert's 7th problem.} Let $a, b \in \overline{\Q},$ with  $a \neq 0, 1$ and $b \not\in \Q.$  Is the value $a^b$ necessarily transcendental? 

\bigskip
A proof that Hilbert's question has an affirmative answer was given independently by Gelfond (1934) and Schneider (1935).  The Gelfond--Schnieder Theorem can be stated equivalently as follows.
Let \[ \sL = \{ x \in \C \colon e^x \in \overline{\Q} \} \]
 denote the $\Q$-vector space of logarithms of algebraic numbers. 
 
 \begin{theorem}[Gelfond--Schneider] If two elements of $\sL$ are linearly dependent over $\overline{\Q}$, then they are linearly dependent over $\Q$.
  \end{theorem}
  
  A fantastic breakthrough was achieved by Alan Baker 30 years later \cite{baker}, when he generalized from two to an arbitrary number of elements of $\sL$.
  
   \begin{theorem}[Baker, 1966] \label{t:baker} If $n \ge 1$ elements of $\sL$ are linearly dependent over $\overline{\Q}$, then they are linearly dependent over $\Q$.
  \end{theorem}
 
 In fact Baker proved an effective refinement of this result giving a strong lower bound on the magnitude of any algebraic linear combination of elements of $\sL$ that are linearly independent over ${\Q}$.  In this paper we will present a proof of a version Baker's Theorem that is slightly weaker than Theorem~\ref{t:baker}.

We next shift our focus from a single linear form in logarithms to arbitrary matrices in $\sL$. The primary conjecture in this direction is the {\bf Structural Rank Conjecture}.   In applications, it is often useful to consider the $\Q$-vector space spanned by $\sL$ and $\Q$, which we denote $\sL+ \Q$.
Given an $m \times n$ matrix $M$ with coefficients in any field of characteristic 0, we define the structural rank of $M$ as follows.
 Choose a $\Q$-basis $\ell_1, \dotsc, \ell_r$ for the coefficients of $M$, and write $M = \sum_{i=1}^r \ell_i M_i$, with $M_i \in M_{m \times n}(\Q)$.  Write $M_x = \sum_{i=1}^r x_i M_i$, where the $x_i$ are indeterminates.  Then $M_x$ is an $m \times n$ matrix with coefficients in the field of rational functions  $F = \Q(x_1, \dots, x_n)$.  We define the {\em structural rank} of $M$ to be the rank of $M_x$ over $F$.  One checks that this definition is independent of the basis $\ell_i$ chosen.

\begin{conjecture}[Structural Rank Conjecture]  The rank of a matrix $M \in M_{m\times n}(\sL + \Q)$ is equal to the structural rank of $M$.
\end{conjecture}

The ``uber conjecture'' in the transcendence theory of special values of logarithms and exponentials of algebraic numbers is the following conjecture of Schanuel.

\begin{conjecture}
Let $y_1, \dotsc, y_n \in \C$ be $\Q$-linearly independent.  Then \[ \trd_{\Q}\Q(y_1, \dots, y_n, e^{y_1}, \dots, e^{y_n}) \ge n. \]  In particular, if $y_1, \dotsc, y_n \in \sL$ are $\Q$-linearly independent, then 
\begin{equation} \label{e:scs}
\trd_{\Q}\Q(y_1, \dots, y_n) = n.
\end{equation}
\end{conjecture}

It is perhaps not surprising that the special case of Schanuel's Conjecture given in (\ref{e:scs}) implies the Structural Rank Conjecture; however an elegant theorem of Roy is that the converse is also true.

\begin{theorem}[Roy]  \label{t:roy}
The Structural Rank Conjecture is equivalent to the special case of Schanuel's conjecture given in (\ref{e:scs}).
\end{theorem}
Theorem~\ref{t:roy} is proven in \S\ref{s:src}.
The strongest unconditional evidence toward the Structural Rank Conjecture is the following theorem of Waldschmidt and Masser.

\begin{theorem}[Waldschmidt--Masser] \label{t:wm}
Let $M \in M_{m \times n}(\sL)$.  Suppose that \[ rank(M)  < mn/(m+n). \]
Then there exist $P \in \GL_m(\Q)$ and $Q \in \GL_n(\Q)$ such that $PMQ = \mat{M_1}{0}{M_2}{M_3}$ 
where the 0 block has dimension $m' \times n' $ with $m'/m + n'/n > 1$.
\end{theorem}

Transcendence results have many important applications in algebraic number theory.  Especially in Iwasawa theory, it is the $p$-adic analogues of these statements that are most relevant.  For example, Leopoldt's conjecture concerns the rank of the matrix of $p$-adic logarithms of a basis of units in a number field $F$.  The $p$-adic analogue of the Waldschmidt--Masser theorem provides the strongest evidence for this conjecture. For instance, for a totally real field $F$ one deduces that the rank of the Leopoldt matrix is at least half the expected one.  We prove the $p$-adic version of the Walsdschmidt--Masser theorem in \S\ref{s:wm}, since the archimedean case is studied more often in the literature.

\bigskip

{\bf Acknowledgements.}  
I would like to extend a great thanks to Damien Roy, who provided a detailed reading of an earlier draft of this paper and made many helpful suggestions that greatly improved the exposition. 
I would also like to thank Mahesh Kakde, my collaborator with whom I learned this material, as well as Michel Waldschmidt for helpful discussions.  
We are very grateful to Alan Harper, whose exposition \cite{harper} we follow for Baker's Theorem, and to Eric Stansifer, whose exposition \cite{stansifer} was very influential for our discussion of the Waldschmidt--Masser Theorem.

This note arose out of a topics course that I taught online at Duke University during Spring 2020, and I thank those attending the  course for lively discussions.

\section{Baker's Theorem}

Before giving an outline of the proof of Baker's Theorem, let us discuss how one could hope to deduce the {\em conclusion} of the theorem.  We are given algebraic numbers $\alpha_1, \dots, \alpha_n \in \overline{\Q}^*$, complex numbers $x_i$ such that $e^{x_i} = \alpha_i$, and a linear dependence
\begin{equation} \label{e:beta}
 \beta_1 x_1  + \cdots + \beta_n x_n = 0 
 \end{equation}
with $\beta_i \in \overline{\Q}$.  
We will show that this implies the existence of integers $\lambda_1, \dots, \lambda_n$, not all zero, such that
\begin{equation}
 \alpha_1^{\lambda_1} \alpha_2^{\lambda_2} \dotsc \alpha_n^{\lambda_n} = 1. 
\label{e:lambda}
 \end{equation}
This implies that the $x_i$, together with the complex number $2 \pi i$, are linearly dependent over $\Q$.
Therefore, the mildly weaker version of Baker's theorem that we will prove is the following:
\begin{theorem} \label{t:weakbaker} If $x_1, \dots x_{n}, 2\pi i \in \sL$ are linearly independent over $\Q$, then 
$x_1, \dotsc, x_{n}$ are linearly independent over $\overline{\Q}$.
\end{theorem}

It does not take much work beyond the methods that we will present to remove $2 \pi i$ and prove the version of Baker's Theorem stated in Theorem~\ref{t:baker} above.  However, to simplify the exposition and highlight the main points we have included $2 \pi i$ in our proof of Theorem~\ref{t:weakbaker}.

Now, how does one deduce the existence of the $\lambda_i$ from the existence of the $\beta_i$?  It may be enticing to try to prove that the $\lambda_i$ can be taken equal to the $\beta_i$, i.e.\ that the $\beta_i$ are rational (or  more generally, that the $\lambda_i$ can somehow be extracted from the $\beta_i$ in a direct way).  However, in practice a more indirect approach is effective. 

\begin{theorem}  \label{t:vand}
Let $\alpha_1, \dots, \alpha_n \in \C^*$.  Suppose there exists a nonzero polynomial 
\[ f(t_1, \dots, t_n) \in \C[t_1, \dots, t_n] \] of  degree $\le L$ in each variable $t_i$ such that 
\[ f(\alpha_1^z, \dots, \alpha_n^z) = 0 \]
for $z =  1, 2, \dots, (L+1)^n  $.  Then there exist integers $\lambda_1, \dots, \lambda_n$, not all zero, such that \[ \alpha_1^{\lambda_1} \alpha_2^{\lambda_2} \dotsc \alpha_n^{\lambda_n} = 1. \]
\end{theorem}

\begin{proof}
Consider the square matrix $M$ whose rows are indexed by the integers \[ z = 1, \dotsc, (L+1)^n  \] and whose columns are indexed by the tuples $\lambda = (\lambda_1, \dots, \lambda_n)$ of integers with $0 \le \lambda_i \le L$, with corresponding matrix entry 
\[ \alpha^{\lambda z} := \alpha_1^{\lambda_1 z} \cdots \alpha_n^{\lambda_n z}. \]
The existence of the polynomial $f$ is equivalent to the existence of a column vector $v$ such that $Mv = 0$.  Indeed, the components of $v$ are precisely the coefficients of $f$.  

The existence of a  nonzero $f$ therefore implies that $\det(M) = 0$.  But $M$ is the Vandermonde matrix associated to the elements $\alpha^\lambda = \alpha_1^{\lambda_1} \cdots \alpha_n^{\lambda_n}$ as the tuple $\lambda$ ranges over all $(L+1)^n$ possibilities.  The vanishing of the determininant therefore implies the existence of two distinct tuples $\lambda, \lambda'$  such that $\alpha^\lambda = \alpha^{\lambda'}$.  We therefore have $\alpha^{\lambda - \lambda'} = 1$ as desired.
\end{proof}

Baker's theorem therefore amounts to using equation (\ref{e:beta}) to  contruct an {\em auxiliary polynomial} $f$ that satisfies the conditions of Theorem~\ref{t:vand}.  We first sumarize Baker's ingenious method to do this.

\begin{enumerate}
\item The Dirichlet Box principle is a method of using the Pigeonhole Principle to construct a polynomial $f$ with certain prescribed zeroes.  One can apply this to the elements $\alpha^z$ appearing in the statement of Theorem~\ref{t:vand}.  Of course, the result will not produce a polynomial with enough zeroes (i.e. we may find zeroes for $z = 1, \dots, A,$ for some $A$, but $A$ will be less than $(L+1)^n$).  Baker's clever insight is that the condition (\ref{e:beta}) allows us to ensure that a certain number of {\em derivatives} of $f$ also have zeroes corresponding to these values of $z$.

\item Baker then proved a complex analytic lemma, which is a quantitative strengthening of the classical Schwarz's Lemma, that shows that the vanishing of $f$ and many of its derivatives implies a strong upper bound on the size of $f$ and half as many of its derivatives, but for $B$ times as many integers $z$ (for some $B > 1$ depending on parameters we will make precise later).

\item Using the fact that the $\alpha_i$ and $\beta_i$ are algebraic, Baker deduces that these bounded values (i.e. the values of $f$ and many of its derivatives for $z = 1, \dots, AB $) must actually be 0.  The basic concept is that an integer of absolute value less than 1 must vanish; a generalization of this elementary statement to algebraic numbers of bounded degree and height is applied.

\item Armed with more vanishing, we now go back to step 2 and once again show that half as many derivatives are small for another factor of $B$ times as many values of $z$.  We iterate this procedure $N$ times until $A B^N > (L+1)^n$, thereby showing that  the auxiliary polynomial $f$ has enough zeroes to apply Theorem~\ref{t:vand}.

\end{enumerate}

In the rest of this section we will describe these steps in detail.

\subsection{Construction of auxiliary polynomial}

For each $\alpha_i$, let $c_i$ denote the leading coefficient of the integral minimal polynomial of $\alpha_i$, and let $d$ denote the maximum degree of any $\alpha_i$.

\begin{lemma}\label{l:c}  There exist integers $a_{i, j, s}$ such that for each integer $j \ge 0$, we have
\[ (c_i \alpha_i)^j = \sum_{s=0}^{d-1} a_{i,j,s} \alpha_i^s. \]
\end{lemma}

\begin{proof} For notational simplicity, we remove the index $i$.  So we consider $\alpha \in \overline{\Q}$ with degree at most $d$, and let $c$ denote the leading coefficient of the integral minimal polynomial of $\alpha$.  
Then there exist integers $b_0, \dotsc, b_{d-1}$ such that
\begin{equation} \label{e:minpoly}
 c \alpha^d = b_{d-1}\alpha^{d-1} + \cdots + b_1 \alpha + b_0. 
 \end{equation}
We prove the result by induction on $j$.  The base cases $0 \le j < d$ are clear.  For $j \ge d$ we assume by induction that there are integers $a_{j-1, s}$ such that
\[ (c \alpha)^{j-1} = \sum_{s=0}^{d-1} a_{j-1, s} \alpha^s. \]
Multiplying by $c \alpha$, we obtain
\begin{equation} \label{e:cj}
 (c \alpha)^j = \left(\sum_{s=0}^{d-2}  c \cdot a_{j-1, s} \alpha^{s+1}\right) + a_{j-1,d-1} (c \alpha^d). \end{equation}
 Plugging in (\ref{e:minpoly}) for $c \alpha^d$ on the right of (\ref{e:cj}), we obtain the desired expression
 \[ (c \alpha)^j = \sum_{s=0}^{d-1}   a_{j, s} \alpha^{s} \]
 where 
 \[ a_{j,s} = \begin{cases} a_{j-1, d-1} b_0 & s=0 \\
 a_{j-1, d-1} b_s + c \cdot a_{j-1, s-1} & 1 \le s \le d-1.
 \end{cases}
 \] 
\end{proof}

Let \[ f(t) = \sum_{\lambda = (\lambda_1, \dotsc, \lambda_n)} p_\lambda t^\lambda := \sum_{\lambda} p_\lambda t_1^{\lambda_1} \cdots t_n^{\lambda_n} \]
be a polynomial of degree $\le L$ in each variable $t_i$.
With the $c_i$ and $a_{i,j,s}$ as in Lemma~\ref{l:c}, we  calculate
\begin{align*}
(c_1\cdots c_n)^{Lz}f(\alpha^z) &= (c_1\cdots c_n)^{Lz}\sum_{\lambda}p_\lambda \alpha^{\lambda z}
\\
&= \sum_{\lambda}p_\lambda c^{Lz - \lambda z}(c\alpha)^{\lambda z} \\
&= \sum_{\lambda}p_\lambda c^{Lz - \lambda z}\prod_{i=1}^n \sum_{s=0}^{d-1} (\alpha_i)^s a_{i, \lambda_i z, s} \\
&= \sum_{s_1, s_2, \dots, s_n=0}^{d-1} \alpha^s \sum_{\lambda} p_\lambda c^{L-\lambda z} \prod_{i=1}^{n} a_{i, \lambda_i z, s}.
\end{align*}
We can therefore force $f(\alpha^z) = 0$ by imposing integer linear conditions on the coefficients $p_\lambda$, namely, that for each $z$ we have \begin{equation} \label{e:integralf}
 \sum_{\lambda} p_\lambda c^{L-\lambda z} \prod_{i=1}^{n} a_{i, \lambda_i z, s} = 0. \end{equation}

This observation allows for the initial construction of an auxiliary polynomial $f$ using the following lemma of Siegel, often known as ``Dirichlet's Box Principle.''

\begin{lemma}[Siegel] 
Let $N > 2M > 0$ be integers, and let $A=(a_{i,j})$ be an $M \times N$ matrix of integers such that 
$|a_{i,j}|<H.$
There is a nonzero vector $b \in \Z^N$ such that $Ab = 0$ and each coordinate of $b$ has absolute value less than $2NH$.
\end{lemma}

\begin{proof}
Consider all vectors $b  \in  \Z^N$ with coordinates of absolute value $\le NH$. There are $(2NH+1)^N> (2NH)^N$ such vectors. For each such $b$, each coordinate of $Ab$ has size at most $(NH)^2$.
The total number of possible vectors $Ab$ is less than $(2(NH)^2)^M$. Since $N > 2M$, the Pigeonhole Principle implies that two distinct $b$ must give the same value of $Ab$. Their
difference gives the desired vector.
\end{proof}

Applying Siegel's lemma to the system of linear equations in (\ref{e:integralf}) will not produce enough zeroes as required by Theorem~\ref{t:vand}.  Indeed, we have not yet used the assumption (\ref{e:beta})!  A key trick noticed by Baker is that it will suffice to have $f$ and sufficiently many of its derivatives vanish.  The precise statement is given below.

\begin{theorem} The following holds for every sufficiently large parameter  $h$.
Let \[ L = h^{2 - 1/(4n)}. \] There exists a polynomial
\[ f(t) = \sum_{\lambda} p_\lambda t^\lambda \in \Z[t_1, \dots, t_n] \]
of degree $\le L$ in each variable $t_i$ such that $|p_\lambda| < e^{h^3}$ and such that the complex analytic function of one variable $z \in \C$ defined by 
 \begin{equation} \label{e:phidef}
\phi(z) = \sum_{\lambda} p_\lambda e^{z(\lambda_1 x_1 + \dotsc + \lambda_n x_n)} 
\end{equation}
 satisfies
\[ \phi^{(m)}(z) = 0 \qquad \text{ for } m = 0, \dots, h^2 - 1, \quad z =  1, \dotsc, h. \]

\end{theorem}

\begin{proof} Note that $\phi(z)$ has been defined so that $\phi(z) = f(\alpha^z) = f(\alpha_1^z, \dotsc, \alpha_n^z)$ for integers $z$.
After dividing the linear dependence (\ref{e:beta}) by $-\beta_n$ (reordering if necessary to ensure this is nonzero) and renaming the coefficients, 
we can write
\[ x_n = \beta_1 x_1 + \cdots + \beta_{n-1} x_{n-1} \]
with $\beta_i \in \overline{\Q}$.  We then have
\begin{equation} \label{e:phidef2}
\phi(z) = \sum_{\lambda} p_\lambda e^{z[(\lambda_1 + \lambda_n\beta_1)x_1 + \dotsc + (\lambda_{n-1} + \lambda_n\beta_{n-1})x_{n-1}]}.
\end{equation}
Note that $\phi^{(m)}(z)$ is the same sum as for $\phi(z)$, but with the term indexed by $\lambda$ multiplied by
\[ \left((\lambda_1 + \lambda_n \beta_1)x_1 + \cdots + (\lambda_{n-1} + \lambda_n \beta_{n-1})x_{n-1} \right)^m. \]  Expanding this out, it suffices to show that
\begin{equation}  \sum_\lambda p_\lambda \alpha^{\lambda z}(\lambda_1 + \lambda_n \beta_1)^{m_1} \cdots (\lambda_{n-1} + \lambda_n \beta_{n-1})^{m_{n-1}} = 0 \label{e:phim}
\end{equation} 
for all tuples of nonnegative integers satisfying \[ m_1 + \cdots + m_{n-1} = m. \]

Let $d$ be the degree of the number field generated by all the $\alpha_i$ and $\beta_i$.
Let $c_i$ denote the leading coefficient in the integral minimal polynomial of $\alpha_i$.
By Lemma~\ref{l:c}, for every nonnegative integer $j$, there exist integers $a_{i,j,s}$ such that
\[ (c_i \alpha_i)^j = \sum_{s=0}^{d-1} a_{i, j, s} \alpha_i^s. \]
 Let $d_i$ and $b_{i,j,s}$ play the same role for the $\beta_i$.

The expression (\ref{e:phim}) will vanish if
\begin{align*}
 \sum_{\mu_1=0}^{m_1} \cdots \sum_{\mu_{n-1}=0}^{m_{n-1}} \sum_{\lambda_1, \cdots \lambda_n = 0}^{L} p_\lambda & \left(\prod_{i=1}^n c_i^{Lz - \lambda_i z} a_{i, \lambda_i z, s_i} \right) \times \\
&  \left( \prod_{j=1}^{n-1} \binom{m_j}{\mu_j} (d_j \lambda_j)^{m_j - \mu_j} \lambda_n^{\mu_j} b_{j, \mu_j, t_j} \right) 
\end{align*}
vanishes for all tuples $(s_1, \dotsc, s_n)$ and $(t_1, \dotsc, t_{n-1})$ with $0 \le s_i, t_i \le d$.

How many linear equations is this in the coefficients $p_\lambda$ we are searching for?  We want vanishing for $0 \le m < h^2$ and $1 \le z \le h$.  Hence the number of such equations is
\[ M = (h^2)^{n-1} h d^{2n-1} = h^{2n-1}d^{2n-1}. \]
Note that $d$ and $n$ are fixed but we are free to make $h$ as large as we want.

The number of variables $p_\lambda$ is $(L+1)^n$, so to ensure this is bigger than $2M$ when $h$ is large, we let $L = h^{2 - 1/(4n)}$ as in the statement of the theorem.
Finally, we bound the size of the coefficients.  An easy induction shows that there is a constant $C$ depending only on the  $\alpha_i$ and $\beta_i$ such that
\[ |a_{i, j, s}| \le C^j \qquad |b_{i, j, t}| \le C^j. \]
Therefore for some constant $K$ depending only on the $\alpha_i$, $\beta_i$, $n$, we have
\[ \prod_{i=1}^n \left| c_i^{Lz - \lambda_i z} a_{i, \lambda_i z, s_i} \right| \le K^{Lz} \le K^{Lh} \]
and similarly
\[  \prod_{j=1}^{n-1} \left| \binom{m_j}{\mu_j} (d_j \lambda_j)^{m_j - \mu_j} \lambda_n^{\mu_j} b_{j, \mu_j, t_j} \right | \le K^{h^2 \log(h)}. \]
By Siegel's Lemma, there is a nontrivial solution in integers $p_\lambda$ such that
\[  |p_\lambda| \le 2K^{Lh+h^2 \log(h)} (L+1)^n \ll e^{h^3}.   \]
\end{proof}

\subsection{Baker's Lemma}

In this section we present a complex analytic lemma of Baker, strengthening the classical Schwarz' Lemma.  This will allow us to bound the sizes of $f$ and many of its derivatives for a multiple $B$ of the $A$ values of $z$ at which we ensured vanishing of our auxiliary polynomial $f$.

\begin{lemma}
 Let $f\colon \C \longrightarrow \C$  be an entire function, let $\epsilon > 0,$ and let $A, B, C, T, U$  be large positive integers such that 
 \begin{equation} \label{e:eps}
  \frac{\epsilon}{2}C > \frac{2T+UAB}{A(\log A)^{1/2}} + \frac{UB A^\epsilon}{\log A}.\end{equation}  Suppose that
\begin{itemize}
\item $|f(z)| \le e^{T + U|z|}$ for  $z \in \C$.
\item $f^{(t)}(z) = 0$ for $t= 0, \dotsc, C-1$ and $z =1, 2, \dotsc, A$.
\end{itemize}
Then \[ |f(z)| \le e^{-(T+Uz)(\log A)^{1/2}} \text{ for } z=1, \dotsc, AB. \] 
\end{lemma}

\begin{proof}
The function
\[ h(z) = \frac{f(z)}{(z-1)^{C} \cdots (z-A)^{C}} \]
is entire by the second assumption.  By the maximum modulus principle on the circle of radius $A^{1+\epsilon}B$ around the origin, we have for $|z| \le AB$:
\[ |h(z)| \le \max_{|w| = A^{1+\epsilon}B} |h(w)|, \]
hence \[ |f(z)| \le \max_{|w| = A^{1+\epsilon}B} |f(w)| \cdot  \max_{|w| = A^{1+\epsilon}B} \left| \frac{(z-1)(z-2) \cdots (z-A)}{(w-1)(w-2)\cdots(w-A)} \right|^{C}, \]
Now
\[  \left| \frac{(z-1)(z-2) \cdots (z-A)}{(w-1)(w-2)\cdots(w-A)} \right| \le \frac{(AB)^A}{(A^{1+\epsilon/2}B)^A} = e^{-(\epsilon/2)A \log A}.\]
Meanwhile
\[ |f(w)| \le e^{T + U A^{1+\epsilon}B}. \]
Our goal is to show that $|f(z)| \le e^{-(T + UAB)(\log A)^{1/2}}$,
so it suffices to show that
\[ -(\epsilon/2)AC \log A + (T + UA^{1+\epsilon}B) \le -(T + UAB)(\log A)^{1/2}.
\]
It is easy to see that this is implied by the assumption of the lemma, 
 \[ \frac{\epsilon}{2}C > \frac{2T+UAB}{A(\log A)^{1/2}} + \frac{UB A^\epsilon}{\log A}.\] 
\end{proof}

Let us now apply Baker's Lemma to our auxiliary polynomial and its derivatives.
With $\phi(z)$  as in (\ref{e:phidef}) and (\ref{e:phidef2}), we have
 \[ \phi^{(m)}(z) = \sum_{\mathclap{{m_i}}} \binom{m}{m_1,\dotsc,m_{n-1}}
  \prod_{i=1}^{n-1} x_i^{m_i} f_{m_1, \dotsc, m_{n-1}}(z)\]
  where 
\begin{equation} \label{e:fdef}
 f_{m_1, \dotsc, m_{n-1}}(z)= \sum_\lambda p_\lambda \alpha^{\lambda z}(\lambda_1 + \lambda_n \beta_1)^{m_1} \cdots (\lambda_{n-1} + \lambda_n \beta_{n-1})^{m_{n-1}}.\end{equation}
It is clear from this last expression that
 \begin{equation} \label{e:fk}
  |f_{m_1, \dotsc, m_{n-1}}(z)| < K^{h^3 + L |z|} \end{equation}
for a suitable constant $K$ depending only on $n$, the $\alpha_i$, and the $\beta_i$.  Indeed, the $p_\lambda$ are bounded by $e^{h^3}$.  The $m_i$ and $\lambda_i$ are bounded by $h^2$.  We choose the constant $K$ so that the inequality (\ref{e:fk}) holds with the $\alpha_i$ or $\beta_i$ replaced by any of their conjugates as well.

Now we apply Baker's Lemma  on each of the functions $f_{m_1, \dotsc, m_{n-1}}(z)$
with
\begin{align*}
 T &= h^3 \log K, \quad  U=L \log K,  \\
 C &= h^2/2  \\
  A &= h, \quad B = h^{1/(8n)}, \quad \epsilon = 1/(8n).
 \end{align*}
We suppose that the constants $K$ and $h$ have been chosen so that the values $T, U, A, B, C$ are integers. Note that the $t$th derivative of $f_{m_1, \dotsc, m_{n-1}}(z)$ for $m_1 + \dots + m_{n-1} \le h^2/2$ and $t \le h^2/2$ is a linear combination of  $f_{m_1', \dotsc, m_{n-1}'}(z)$ with $m_1' + \dots + m_{n-1}' \le h^2$, which gives the desired vanishing for $z = 0, \dots, h$ by the construction of the polynomial $f$.
 
Furthermore, with our selection of parameters, the required inequality (\ref{e:eps}) reads
\[ \frac{h^2}{32n} > \frac{2(\log K)h^3 + (\log K) h^{2-1/(4n)} \cdot h \cdot h^{1/(8n)}}{h(\log h)^{1/2}} + \frac{(\log K)h^{2 - 1/4n} \cdot h^{1/(8n)} \cdot h^{1/(8n)}}{\log h}, \]
which is easily seen to hold for $h$ large.  Baker's Lemma therefore yields 
\begin{equation} \label{e:fineq}
 |f_{m_1, \dotsc, m_{n-1}}(z)| < K^{-(h^3 + Lz)(\log h)^{1/2}} \end{equation}
for $m_1 + \dotsc + m_{n-1} \le h^2/2$ and $z=1, \dotsc, h^{1+1/(8n)}$.

\subsection{Discreteness of Algebraic Integers}

We apply the following elementary basic principle.

\begin{lemma} \label{l:discrete}
Suppose that $a \in \overline{\Q}$ such that $d a$ is an algebraic integer for some positive integer $d$.
Suppose that $|a| < \epsilon$ for some positive real number $\epsilon$ and that every conjugate $\sigma(a)$ satisfies $|\sigma(a)| < M$ for some positive real number $M$.  Finally suppose that $[{\Q}(\alpha) \colon {\Q}] \le n$ and that $\epsilon M^{n-1}d^n   < 1$. Then $a = 0$.
\end{lemma}

\begin{proof}
We bound the norm of the algebraic integer $da$:
\begin{align*} | \N_{\Q(a)/\Q}(da)| &= \prod_{\sigma \colon \Q(a) \hookrightarrow \C} |\sigma(da)| \\
& \le d^n | a| \cdot \prod_{\sigma \neq 1} | \sigma(a) | \\
& \le d^n \epsilon M^{n-1} < 1.
\end{align*}
An integer of absolute value less than 1 must be 0.  Hence $\N_{\Q(a)/\Q}(da) = 0$, so $a=0$.
\end{proof}

We apply Lemma~\ref{l:discrete} to each of the algebraic numbers $f_{m_1, \dotsc, m_{n-1}}(z)$ defined in (\ref{e:fdef}) for 
 \[ m_1 + \dotsc + m_{n-1}\le h^2/2,  \qquad z=0,1, \dotsc, h^{1+1/(8n)}. \] In (\ref{e:fineq}) we showed that
 \begin{equation} \label{e:ek}
   |f_{m_1, \dotsc, m_{n-1}}(z)| < \epsilon:= K^{-(h^3 + Lz)(\log h)^{1/2}}. \end{equation}
 We also saw in (\ref{e:fk}) that
\begin{equation}
  |\sigma(f_{m_1, \dotsc, m_{n-1}}(z))| < M := K^{h^3 + L z}  \label{e:mdef}
  \end{equation}
for each $\sigma$.  It is easy to see from its definition that the denominator of $f_{m_1, \dotsc, m_{n-1}}(z)$ can be cleared by an integer of size at most \begin{equation} 
 d := K^{h^2 + Lz},  \label{e:ddef}
  \end{equation}
 after making $K$ larger if necessary depending only on the $\alpha_i$ and $\beta_i$.

The inequality $\epsilon M^{n-1}d^n   < 1$ is then easily seen to hold for $h$ large because of the extra factor $(\log h)^{1/2}$ in the exponent of (\ref{e:ek}), so we conclude that
\[ f_{m_1, \dotsc, m_{n-1}}(z) = 0 \] for 
 $m_1 + \dotsc + m_{n-1}\le h^2/2$ and $z=0,1, \dotsc, h^{1+1/(8n)}$.

\subsection{Bootstrapping}

We repeat the process described above over and over, in the $k$th iteration using Baker's Lemma on the functions $f_{m_1, \dots, m_{n-1}}(z)$ with $m_1 + \dots + m_{n-1} \le h^2/2^{k+1}$ with the parameters:
\begin{align*}
 T &= h^3 \log K, \quad  U=L \log K,  \\
 C &= h^2/2^{k+1}  \\
  A &= h^{1+k/(8n)}, \quad B = h^{1/(8n)}, \quad \epsilon = 1/(8n).
 \end{align*}
We assume that $K$ and $h$ had been chosen initially so that the quantities above are integers.   In the $k$th iteration we obtain that  
\[ |f_{m_1, \dotsc, m_{n-1}}(z)| <  K^{-(h^3 + Lz)(\log(h+k/(8n)))^{1/2}} \]
for 
 \[ m_1 + \dotsc + m_{n-1}\le h^2/2^{k+1},  \qquad z=0,1, \dotsc, h^{1+(k+1)/(8n)}. \]
 The quantities in (\ref{e:mdef}) and (\ref{e:ddef}) do not change, so again Lemma~\ref{l:discrete} implies that the values 
 $f_{m_1, \dotsc, m_{n-1}}(z)$ vanish.  We may therefore move on to the next $k$.
 
 Each iteration multiplies the number of zeroes by $B = h^{1/(8n)}$. After $16n^2$ iterations we will  obtain more than $h^{2n}$ zeroes.  Since $L = h^{2 - 1/(4n)}$ and $h$ is large, we have 
 \[ h^{2n} > (L+1)^n, \]
 so the polynomial $f = f_{0,0,\dots, 0}$ satisfies the conditions of Theorem~\ref{t:vand}.  Therefore there exist integers $\lambda_1, \dots, \lambda_n$, not all zero, such that
 \[ \alpha_1^{\lambda_1} \cdots \alpha_n^{\lambda_n} = 1. \]
 This completes the proof of Baker's Theorem.
 
 \section{Ax's Theorem} \label{s:ax}
 
Moving on from linear forms in elements of $\sL$ to arbitrary polynomials, we remind the reader of Schanuel's Conjecture that was stated in the introduction.

\bigskip

\noindent
{\bf Schanuel's Conjecture.}  Let $y_1, \dots, y_n \in \C$ be $\Q$-linearly independent.  Then \[ \trd_\Q \Q(y_1, \dotsc, y_n, e^{y_1}, \dotsc, e^{y_n}) \ge n. \]

While little is known about this conjecture, we have the following function field analogue proved by James Ax \cite{ax}.

\begin{theorem}[Ax] \label{t:ax}
 Let $y_1, \dots, y_n \in t\C[[t]]$ be $\Q$-linearly independent.  Then \[ \trd_{\C(t)} \C(t)(y_1, \dotsc, y_n, e^{y_1}, \dotsc, e^{y_n}) \ge n.  \]
 \end{theorem}

In this section we prove Ax's theorem.  The section is self-contained and may be skipped by readers not interested in the function field setting.

\subsection{Derivations}

\begin{definition}
Let $A$ be a commutative ring and $B$ a commutative $A$-algebra.  An $A$-derivation of $B$ into a $B$-module $M$ is an $A$ linear map 
\[ D \colon B \longrightarrow M \]
such that \begin{equation} \label{e:derdef}
 D(ab) = a D(b) +  D(a)b, \qquad a, b \in B, \end{equation}
where we view $M$ as both a left and right $B$-module since $B$ is commutative.
\end{definition}
\bigskip

There is a pair $(d = d_{B/A}, \Omega_{B/A})$ of a $B$-module $\Omega_{B/A}$ and an $A$-derivation \[ d\colon B \longrightarrow \Omega_{B/A} \] that is universal in the sense that any $A$-derivation $D \colon B \lra M$ can be obtained by composing $d_{B/A}$ with a $B$-module homomorphism $\Omega_{B/A} \lra M$.
 The module of {\em K\"ahler differentials} $\Omega_{B/A}$ is defined as the quotient of the free $B$-module generated by formal generators $db$ for each $b \in B$ by the relations
$da = 0$ for $a \in A$, $d(b+b') = db + db'$, and $d(bb') = b \cdot db' + b' \cdot db$.  The universal derivation $d_{B/A} \colon B \lra \Omega_{B/A}$ is defined by $d_{B/A}(b) = db$.

\begin{lemma} \label{l:algebraic}
Let $F/K$ be field extension and $x \in F$ separable algebraic over $K$. Then $dx = 0$ in $\Omega_{F/K}$.
\end{lemma}

\begin{proof}
Let $f(x) \in K[x]$ be the minimal polynomial of $x$.  Then 
\[ 0 = d(f(x)) = f'(x) dx. \]
Since $x$ is separable, $f'(x) \neq 0$, so $dx = 0$.
\end{proof}

Meanwhile, if $F(t)$ denotes the function field in one variable over the field $F$, we have that $\Omega_{F(t)/F}$ is the 1-dimensional $F(t)$-vector space generated by $dt$, with $d(f(t)) = f'(t) dt$.

\begin{lemma} \label{l:extend}
 Let $K \subset F \subset L$ be fields of characteristic 0.  Let $D \colon F \longrightarrow F$ be a $K$-derivation.  Then $D$ can be extended to a $K$-derivation $L \longrightarrow L$.
 \end{lemma}
 
\begin{proof}
For $f \in F[t]$, let $f^D$ denote the polynomial where $D$ has been applied to the coefficients of $f$.
We show how to extend the derivation $d$.
 Let $z \in L, z \not \in F$.  If $z$ is algebraic over $F$, let $p(x)$ be its minimal polynomial.  Define
\begin{equation} \label{e:rule1} u = - \frac{p^D(z)}{p'(z)}, \quad \tilde{D}(g(z)) = g^D(z) + g'(z)u. \end{equation}

If $z$ is transcendental over $F$, we define
\begin{equation} \label{e:rule2}
 \tilde{D}(g(z)) = g^D(z) + g'(z)u 
 \end{equation}
  for any $u \in L$.
We leave it to the reader to check that setting $\tilde{D}|_F = D$ and using (\ref{e:rule1}) or (\ref{e:rule2}) to extend to $F(z)$ yields a derivation $\tilde{D}$.  Now one uses Zorn's Lemma to extend $D$ all the way to $L$. 
\end{proof}

\begin{corollary} \label{c:trd}
 Let $K \subset L$ be fields of characteristic 0. Then \[ \dim_L \Omega_{L/K} = \trd_K L.\]  More generally, if $K \subset F \subset L$, then
\[ \dim_L (L \cdot d_{L/K}(F)) = \trd_K F. \]
\end{corollary}

\begin{proof}
  Let $f_1, \dots, f_n$ be a transcendence basis for $F/K$.  Suppose that
\[ \sum_{i=1}^{n} a_i d_{L/K}f_i = 0 \] with $a_i \in L$.  
By the universal property of $d_{L/K}$, we have 
\[ \sum_{i=1}^{n} a_i D(f_i) = 0 \]
for any $K$-derivation $D_i \colon L \rightarrow L$.
Therefore, if we show that for each $i$ there exists a $K$-derivation $D_i \colon L \rightarrow L$  such that $D_i(f_j) = \delta_{ij}$, then we obtain $a_i =0$ for all $i$.  This yields the linear independence of the $d_{L/K} f_i$ over $L$.

The existence of the $D_i$ follows from the proof of Lemma~\ref{l:extend}.  We can first extend the 0 derivation on $K$ to $K(f_i)$ setting $z = f_i$, $u = 1 - f_i$ in 
(\ref{e:rule2}), and then inductively extend to $K(f_1, \dots, f_n)$ by setting $z = f_j, u = - f_j$ for $j \neq i$.  Finally we extend $D_i$ to $L$ using Lemma~\ref{l:extend} once more.
\end{proof}

\subsection{Derivation on Kahler differentials}

Let $A$ be a commutative ring and $B$ an $A$-algebra.
Let $D\colon B \longrightarrow B$ be a derivation such that $D(A) \subset A$.  There exists an $A$-linear map
\[ D^1 \colon \Omega_{B/A} \longrightarrow \Omega_{B/A} \]
satisfying
\begin{equation} \label{e:d1def}
D^1(f dg) = (Df) dg + f d(Dg). \end{equation}
We leave the verification of this to the reader, but we note that a more general fact is true.  If we consider 
the graded algebra of differentials
\[ \Omega^*_{B/A} = \bigoplus_{n=0}^{\infty} \ \bigwedge\nolimits_B^n \Omega_{B/A}, \] then the differential 
$D\colon B \longrightarrow B$ extends to a graded derivation \[ D^* \colon \Omega^*_{B/A} \longrightarrow \Omega^*_{B/A} \]
satisfying (\ref{e:derdef}), where the 0th graded piece is $D$ and the 1st graded piece is $D^1$.  In our proof of Ax's Theorem, we will only need the map $D^1$, but let us note that the rule (\ref{e:d1def}) generalizes: for any $f \in B$ and $\omega \in \Omega_{B/A}$, we have 
\begin{equation} \label{e:lrd1}
 D^1(f \omega) = (Df) \omega + f D^1(\omega). 
 \end{equation}

\begin{lemma} \label{l:yz}
 Let $y \in t \C[[t]], \ z = e^y,$ and let  $D \colon \C((t)) \longrightarrow \C((t))$ be a 
$\C$-derivation of the form $D(f(t)) = f'(t) \cdot g(t)$ for some fixed $g(t) \in \C((t))$.  
Then \[ D^1(dy - z^{-1}dz) = 0 \]
in $\Omega_{\C((t))/\C}$.
\end{lemma}

\begin{proof}
A direct computation with the definition (\ref{e:d1def}) shows that in general, we have
\[ D^1(dy - z^{-1}dz) = d(Dy - z^{-1} Dz). \]
Yet when $z = e^y$, the term $Dy - z^{-1} Dz$ vanishes for the derivation  $D(f(t)) = f'(t) \cdot g(t)$.  The result follows.
\end{proof}

\begin{lemma} \label{l:kdk}
Let $K \subset L$ be fields, $D\colon L \longrightarrow L$ a derivation such that $\ker D = K$.
Then the map 
\[ L \otimes_K \ker D^1 \longrightarrow \Omega_{L/K}, \qquad f \otimes \omega \mapsto f\omega, \]
is injective.
\end{lemma}

\begin{proof}
 Suppose there exist \[ f_1, \dots, f_n \in L^*, \qquad \omega_1, \dots, \omega_n \in \ker D^1 \] such that \begin{equation} \label{e:sfw}
 \sum_{i=1}^{n} f_i \otimes \omega_i \mapsto 0, \qquad \text{i.e. } \sum_{i=1}^{n} f_i \omega_i =0.
 \end{equation}
 Scale so that $f_1 = 1.$
If all the $f_i$ lie in $K$, 
then 
\[  \sum_{i=1}^{n} f_i \otimes \omega_i = 1 \otimes  \sum_{i=1}^{n} f_i  \omega_i = 1 \otimes 0 = 0, \]
so we are done.  Suppose this is not the case and take the minimal such vanishing linear combination. 
By minimality, we can assume that the $\omega_i$ are linearly independent over $L$.
 Apply $D^1$ to the expression (\ref{e:sfw}).  Using (\ref{e:lrd1}) we find
   \[ 0 = \sum_{i=1}^n ((Df_i) \omega_i + f_i D^1(\omega_i)) =  \sum_{i=1}^n (Df_i) \omega_i, \]
   where the second equality holds since $\omega_i \in \ker D^1$.  By the linear independence of the $\omega_i$ over $L$, we see that $Df_i = 0$ for all $i$ and hence by assumption $f_i \in K$ for all $i$.  This gives the desired result.
\end{proof}

The following is a technical algebraic lemma that will allow us to reduce to the setting of function fields of curves.

\begin{lemma} \label{l:w} Let $K \subsetneq L$ be an extension of fields with $K$ relatively algebraically closed in $L$.  Let 
\begin{align*}
 W = \{ & F\colon K \subset F \subset L, \ \  \trd_F L = 1, \\
& F \text{ relatively algebraically closed in } L \}.
\end{align*}
Then \[ \bigcap_{F \in W} F = K. \]
\end{lemma}

\begin{proof}
 Let $t \in L$, $t \not\in K$.  We need to show there exists $F \in W$ such that $t \not \in F$.  Since $K$ is relatively algebraically closed in $L$, the element $t$ is transcendental over $K$.

Choose a transcendence basis for $L/K$ consisting of $t$ and a set $B$ of elements not in $K(t)$.  Let $F$ be the relative algebraic closure of $K(B)$ in $L$, i.e. 
\[ F = \{ x\in L\colon x \text{ algebraic  over } K(B)\}. \]
Then $F \in W$, since $L$ is algebraic over $F(t)$.  Since $t \not\in F$, this completes the proof. \end{proof}

In some sense, the following lemma is the main engine of Ax's proof.

\begin{lemma} \label{l:kmap}
Let $L/K$ be fields of characteristic 0. Denote by
$dL$ the $K$-subspace of  $\Omega_{L/K}$  spanned by $df$ for $f \in L$.  Denote by
$dL/L$ the $\Z$-submodule of $\Omega_{L/K}$ spanned by $f^{-1} df$ for $f \in L^*$. 
Then the canonical map of $K$-vector spaces 
\begin{equation} \label{e:kmap}
 K \otimes_{\Z} dL/L \longrightarrow \Omega_{L/K}/dL, \qquad k \otimes \frac{df}{f} \mapsto \frac{k}{f}df, 
 \end{equation}
is injective, where $\Omega_{L/K}/dL$ denotes the quotient of $\Omega_{L/K}$ by the $K$-subspace spanned by $df$ for $f \in L$. 
\end{lemma}

\begin{proof}
  Choose an element \begin{equation} \label{e:ourelement}
  \sum_{i=1}^{n} k_i \otimes f_i^{-1}df_i
  \end{equation} in the kernel of the map (\ref{e:kmap}), with $n$ minimal.  By minimality, the $k_i$ are linearly independent over $\Q$.  We will show that each $f_i$ lies in $\overline{K}_L$, the relative algebraic closure of $K$ in $L$.  By Lemma~\ref{l:algebraic}, this will imply $df_i =0$, giving the desired injectivity.

If $L = \overline{K}_L$, there is nothing to prove.  Otherwise let $F \in W$, with $W$ as in Lemma~\ref{l:w}.   So $K \subset F \subset L$, $\trd_F L =1$, and $\overline{F}_L = F$.  Since the element (\ref{e:ourelement}) lies in the kernel of (\ref{e:kmap}), we have
\begin{equation} \label{e:kf}
 \sum_{i=1}^n k_i f_i^{-1} df_i = \sum_{i=1}^m k_i' df_i'. \end{equation}
for some $f_i' \in L^*$, $k_i' \in K$.  Now, we would like to use  properties of function fields of curves, but unfortunately, we do not know that $L$ is finitely generated over $F$.  To this end, we consider a field $L'$ generated over $F$ by the $f_i$, the $f_i'$, and by 
any elements used in any relations in $\Omega_{L/K}$ used to obtain equation (\ref{e:kf}).  The field $L'$ then still has transcendence degree 1 over $F$, and is finitely generated over $F$.  We may therefore identify $L'$ with the function field of a smooth projective algebraic curve over $F$.  
Furthermore, the equation (\ref{e:kf}) still holds in $\Omega_{L'/K}$ by construction, and so it holds also in $\Omega_{L'/F}$.

Points $P$ on this curve correspond to valuations
\[ \ord_P \colon (L')^* \longrightarrow \Z. \]
Associated to $P$ we also have a residue map
\[ \res_P \colon \Omega_{L'/F} \longrightarrow F. \]
The residue and valuation map satisfy the following well known properties.  
For all $g \in (L')^*$, we have \[ \res_P(g^{-1}dg) = \ord_P(g), \qquad   \res_P(dg) = 0. \]
Applying $\res_P$ to  (\ref{e:kf}), we get
\[ \sum_{i=1}^{n} k_i \ord_P(f_i) = 0. \]
By $\Q$-linear independence of the $k_i$, we obtain $\ord_P(f_i) = 0$ for all $P, i$.  But a function on a smooth projective curve with no zeroes or poles must be constant, and hence $f_i \in F$ for all $i$.
Since this holds for all $F$, we have by Lemma~\ref{l:w} that $f_i \in \overline{K}_L$.  This is the desired result.
\end{proof}

We can now complete the proof of Ax's Theorem.
\begin{proof}[Proof of Theorem~\ref{t:ax}]
Let $y_1, \dots, y_n \in t \C[[t]]$ and write $z_i = e^{y_i} \in \C[[t]].$
   Let \[ L = \C(y_1, \dots, y_n, z_1, \dots, z_n). \]  It suffices to show that if $\trd_\C L \le n$, 
   then $y_1, \dots, y_n$ are $\Q$-linearly dependent.
Suppose $\trd_\C L \le n$.  Then by Corollary~\ref{c:trd}, the differentials \[ \omega_i = dy_i - z_i^{-1} dz_i \in \Omega_{L/\C} \]
 for $i=1, \dots, n$ together with $dy_1$ must be linearly dependent over $L$:
\begin{equation} \label{e:fwgy}
 \sum_{i=1}^n f_i \omega_i + g dy_1 = 0, 
 \end{equation}
with $f_i, g \in L$ not all zero.  Note that if $y_1'(t) =0$, then $y_1$ is a constant, and since $y_1 \in t \C[[t]]$ we would get $y_1 =0$.  Then the $y_i$ are trivially linearly dependent; so we may assume hereafter that $y_1'(t) \neq 0$.
Define a $\C$-derivation \[ D\colon L \longrightarrow L, \qquad D(f(t)) = f'(t)/y_1'(t). \]
By Lemma~\ref{l:yz}, we have $D^1(\omega_i)=0$.
Furthermore, a direct computation shows
\[ D^1(d y_1) = d(Dy_1) = d(1) = 0. \]

Therefore, we have that
\[ \sum f_i \otimes \omega_i + g \otimes dy_1 \in \ker((L\otimes_\C \ker D^1) \longrightarrow \Omega_{L/\C}). \]
By Lemma~\ref{l:kdk}, we may assume $f_i, g \in \C$.

Rewrite the equation $\sum f_i \omega_i + g dy_1 = 0$ in the form
\[ \sum_{i=1}^{n} f_i \cdot ( - z_i^{-1} dz_i) = -\sum_{i=1}^n f_i d y_i - g dy_1. \]
Lemma~\ref{l:kmap} implies that either all $f_i = 0$, or the $z_i^{-1} dz_i$ are $\Q$-linearly dependent.
In the first case, from (\ref{e:fwgy}) and the fact that the $f_i$, $g$ are not all zero we would get $dy_1 = 0$.  Hence $y_1$ is a constant, and as noted earlier this implies that $y_1 =0$.  Therefore we suppose we are in the second case, say
\[ \sum m_i (dz_i)/z_i = 0 \]
with $m_i \in \Z$ not all zero.  This implies implies
\[ d\left(\prod z_i^{m_i}\right)/ \left(\prod z_i^{m_i} \right)= 0, \]
so $\prod z_i^{m_i} = e^{\sum m_i y_i}$ is a constant.  By considering constant terms, this constant must be 1. Therefore
\[ \sum m_i y_i = 0, \]
giving the desired linear dependence of the $y_i$ over $\Q$. 
\end{proof}
 
 \section{The Structural Rank Conjecture} \label{s:src}
 
 We now return to the classical setting over $\C$, rather than the function field setting, and move on to consider matrices of elements of $\sL$.  The simplest case of $2 \times 2$ matrices leads to the following {\bf Four Exponentials Conjecture}.
 
 \begin{conjecture} Let $M \in M_{2\times 2}(\sL)$.  Then $\det(M) =0$ only if the rows or columns of $M$ are linearly dependent over $\Q$.
 \end{conjecture}
 
 This conjecture was first stated in print in 1957 by Schneider \cite{schneider}, though versions had been considered over the previous two decades by Selberg, Siegel, Alaoglu--Erdos \cite{ae}, and others.   It remains wide open.  The strongest theoretical evidence for the conjecture is the following {\bf Six Exponentials Theorem}.
 
  \begin{theorem} \label{t:six}  Let $M \in M_{2\times 3}(\sL)$.  Then $\rank(M) < 2$ only if the rows or columns of $M$ are linearly dependent over $\Q$.
 \end{theorem}
 
 The Six Exponentials Theorem was proven independently by Lang \cite{lang} and Ramachandra \cite{ram} in the late 1960s.  
 See Waldschmidt's delightful personal account \cite{fexp} for details and references on the history of the Four Exponentials Conjecture and the Six Exponentials Theorem.
 
 The Six Exponentials Theorem follows as a special case of the theorem of Waldschmidt--Masser that we will discuss later in this paper.  A naive generalization of the Four Exponentials Conjecture to matrices of arbitrary dimension does not hold---in general, matrices may have lower than maximal rank even if the rows and columns are linearly independent over $\Q$.  As an example, note that
 \[ \det \begin{pmatrix}
 x & z & 0 \\
0 & y & -x \\
y & 0 & z 
\end{pmatrix} = 0.
 \]
Therefore, if we substitute for $x$, $y$, and $z$ any elements of $\sL$ that are linearly independent, then we obtain a matrix of rank $< 3$ whose rows and columns are linearly independent over $\Q$.   Examples such as these motivate the {\bf Structural Rank Conjecture} that was stated precisely in the introduction.  The matrix above has structural rank equal to 2.
 
 \subsection{The $p$-adic setting}
 
 Most statements in transcendence theory have analogs in the $p$-adic setting.  As we will describe below, these analogs are particularly important in Iwasawa theory.  Let $p$ be a prime number, and let
$\C_p = \hat{\overline{\Q}}_p$ denote the completion of the algebraic closure of $\Q_p$.
The statements below work equally well over $\Q_p$, but working with $\C_p$ provides extra generality.
There exist a $p$-adic logarithm and a $p$-adic exponential function
\begin{equation}
\begin{aligned} \label{e:logp}
 \log_p &\colon \{ x \in \C_p \colon |x - 1| < 1 \} \longrightarrow \C_p  \\
\exp_p & \colon \{ x \in \C_p \colon |x| < p^{-1/(p-1)}\} \longrightarrow \C_p \end{aligned}
\end{equation}
defined by the usual power series
\[  \log_p(1-x) = -\sum_{n=1}^{\infty} \frac{x^n}{n},  \qquad \exp_p(x) = \sum_{n=1}^{\infty} \frac{x^n}{n!}.  \]
The functions $\log_p$ and $\exp_p$ are injective group homomorphism on the domains given in (\ref{e:logp}).
The $p$-adic logarithm extends uniquely to a continuous homomorphism
\[ \log_p \colon \{ x \in \C_p \colon |x| = 1 \} \longrightarrow \C_p \]
since every $x \in \C_p$ with $|x| = 1$ satisfies $|x^n - 1| < 1$ for an appropriate positive integer $n$, and we may define $\log_p(x) = \frac{1}{n}\log_p(x^n)$.  Next we extend $\log_p$ to a continuous homomorphism
\[ \log_p \colon \C_p^* \longrightarrow \C_p \]
by fixing Iwasawa's (noncanonical) choice $\log_p(p) =0$.  The kernel of $\log_p$ on $\C_p^*$ then consists of elements of the form $p^a \cdot u$ where $a \in \Q$ and $u$ is a root of unity.

We  define the $\Q$-vector space of $p$-adic logarithms of algebraic numbers:
\[ \sL_p = \{ \log_p(x) \colon x \in \overline{\Q}^*\} \subset \C_p. \]
The $p$-adic version of Baker's Theorem was proved by Brumer following Baker's method.

\begin{theorem}[Baker--Brumer]   \label{t:brumer}
 Let $y_1, \dots, y_n \in \sL_p$ be linearly independent over $\Q$.  Then $y_1, \dots, y_n$ are linearly independent over $\overline{\Q}$.
 \end{theorem}

Similarly there are natural analogs of Schanuel's Conjecture and the Structural Rank Conjecture in the $p$-adic setting.  To be precise we state the latter of these:

\begin{conjecture}[$p$-adic Structural Rank Conjecture] Let \[ M \in M_{m \times n}(\sL_p + \Q) \subset M_{m \times n} (\C_p). \] The rank of $M$ is equal to the structural rank of $M$.
\end{conjecture}

\subsection{Applications in Number Theory}

Statements in transcendence theory have important applications in algebraic number theory.  In this section, we describe two important conjectures in Iwawasa theory that are special cases of the $p$-adic Structural Rank Conjecture.
These conjectures are our personal motivation for this study.

\subsubsection{Leopoldt's Conjecture}
Fix a prime $p$ and an embedding $\overline{\Q} \hookrightarrow \C_p$. 

\begin{conjecture}[Leopoldt's Conjecture]
  Let $F$ be a number field of degree $n$ over $\Q$ and let $\sigma_1, \dots, \sigma_n$ denote the embeddings $F \hookrightarrow \overline{\Q}$.  Let $u_1, \dots, u_r$ be a $\Z$-basis for $\cO_F^*/\mu(F)$.  Let 
\[ M = (\log_p\sigma_j(u_i)) \in M_{r \times n}(\sL_p). \]
Then $\rank_{\C_p}(M) = r$.
\end{conjecture}

\begin{prop} \label{p:l}
The $p$-adic structural rank conjecture implies Leopoldt's Conjecture.
\end{prop}

\begin{proof}  The important point here is that the archimedean analog of the statement of Leopoldt's Conjecture is known to be true; this is the classical nonvanishing of the regulator of a number field.  More precisely, if we 
fix an embedding $\overline{\Q} \hookrightarrow \C$ and let 
$N = \left( \log |\sigma_j(u_i)| \right)$, where the absolute value denotes the usual absolute value on $\C$, then we have 
\[ \rank_{\C}(N) = r. \]
This is proved using the fact that $\log | \cdot |$ takes values in the ordered field $\R$ (whereas $\log_p$ does not).  For this reason, the $p$-adic statement lies far deeper than the archimedean one.

The field $\Q(x_1, \dots, x_k)$ appearing in the definition of the structural rank provides a bridge between the $p$-adic and complex settings, with the $p$-adic Structural Rank Conjecture doing most of the heavy lifting.

Let $\{c_1, \dots, c_k\} \subset \{\sigma_j(u_i)\}$ 
such that $\{\log_p(c_i) \}$ is a $\Q$-basis for the $\Q$-vector space spanned by the 
$ \log_p(\sigma_j(u_i)) $.  Write 
\[ M = (\log_p \sigma_j(u_i)) = \sum_{i=1}^{k} M_i \log_p(c_i) \]
with $M_i \in M_{r \times n}(\Q)$.   
The $p$-adic Structural Rank Conjecture implies that
\begin{align}
 \rank_{\C_p} M &= \rank_{\Q(x_1, \dots, x_k)}\left( \sum_{i=1}^k M_i x_i \right) \nonumber \\
& \ge \rank_{\C} \left( \sum_{i=1}^k M_i \log |c_i| \right) \nonumber \\
& = \rank_{\C} \left( \log |\sigma_j(u_i)| \right) \label{e:applylog} \\
& = r. \nonumber
\end{align}
Hence $\rank_{\C_p} M  \ge r$, and so we must have equality. Note that in the equality (\ref{e:applylog}), we are implicitly using the fact that if $u_i \in \overline{\Q}^*$ are $p$-adic units, and $m_i \in \Z$ are integers, then
\[ \sum m_i \log_p(u_i) = 0 \Longrightarrow \prod u_i^{m_i} \text{ is a root of unity } \Longrightarrow \sum m_i \log |u_i| = 0. 
\]
\end{proof}

Let us describe two applications of Leopoldt's conjecture.

\bigskip

{\bf Algebraic (Iwasawa Theory).}  By class field theory, Leopoldt's Conjecture implies that the maximal pro-$p$ abelian extension of $F$ unramified outside $p$ has $\Z_p$-rank equal to $r_2 + 1$, where $2r_2$ is the number of embeddings $F \hookrightarrow \C$ with image not contained in $\R$. 

\bigskip

{\bf Analytic ($p$-adic $L$-functions).}  Let $F$ be a totally real field, so \[ r = \rank \cO_F^* = [F : \Q] - 1. \]  There is a $p$-adic analog of the classical Dedekind zeta function of  $F$ denoted $\zeta_{F,p}$.  A theorem of Colmez \cite{colmez} states that
\[ \lim_{s \rightarrow 1} (s-1) \zeta_{F, p}(s) = (*) R_p(F), \]
where \[ R_p(F) = \det(\log_p(\sigma_j(u_i))_{i,j=1,\dots, r} \]
and $(*)$ denotes a specific nonzero algebraic number that we do not describe precisely here.
This is a $p$-adic ``class number formula."  Therefore $\zeta_{F, p}(s)$ has a pole at $s=1$ if and only if Leopoldt's Conjecture is true.

\subsubsection{The Gross--Kuz'min Conjecture}

There is an analog of Leopoldt's Conjecture due independently to Gross and Kuz'min concerning $p$-adic $L$-functions at $s=0$ rather than $s=1$.  Unlike the case of classical $L$-functions, there is no functional equation for $p$-adic $L$-functions relating the values at 0 and 1.

We refer the reader to Gross's article \cite{gross} for details about the Gross--Kuz'min conjecture beyond what we write below.
  To state the conjecture, 
let $H$ be a CM field and $H^+$ its maximal totally real subfield.  Let
\[ U_p^- = \{ u \in H^* \colon |u|_w = 1 \text{ for all } w \nmid p \}. \]
Here $w$ ranges over all places of $H$ that do not divide $p$, including the archimedean ones.
Then $\rank(U_p^-) = r$, where $r$ is the number of primes of $H^+$ above $p$ that split completely in $H$. 

Let $X_p$ denote the $\C_p$-vector space with basis indexed by the places of $H$ above $p$.  Let $c$ denote the non-trivial element of $\Gal(H/H^+)$, i.e.\ $c$ = complex conjugation.  Let $X_p^-$ denote the largest quotient of $X_p$ on which $c$ acts as $-1$.  Then $X_p^-$ has dimension $r$.  We define two maps
\[ \ell_p, o_p  \colon U_p^- \longrightarrow X_p^-. \]
The coordinate of $o_p(u)$ at the component corresponding to a place $\fP$ of $H$ is $\ord_{\fP}(u)$, and the coordinate of $\ell_p(u)$ is 
$\log_p( \N_{H_{\fP}/\Q_p}(u)).$  We extend $\ell_p$ and $o_p$ to $\C_p$-linear maps \[ U_p^- \otimes \C_p \longrightarrow X_p^-. \]
 It is not hard to show using Dirichlet's Unit Theorem that  $o_p$ is an isomorphism, and we define
 \[ R_p^-(H) = \det(\ell_p \circ o_p^{-1}). \]

\begin{conjecture}[Gross--Kuz'min]
We have $R_p^-(H) \neq 0$.
\end{conjecture}
 A  proof similar to the proof of Proposition~\ref{p:l} shows that the $p$-adic Structural Rank Conjecture implies the Gross--Kuz'min conjecture (one uses $\ord_p$ in place of $\log |\cdot |$).  Once again there are algebraic and analytic interpretations of this conjecture.
 
 \bigskip

{\bf Algebraic (Iwasawa Theory).}  By class field theory, the Gross--Kuz'min conjecture implies a bound on the growth of the $p$-parts of class groups of fields in the cyclotomic $\Z_p$-extension of $H$. See \cite{fg}*{Prop. 3.9} for details.

\bigskip

{\bf Analytic ($p$-adic $L$-functions).}  Let $\chi$ denote the nontrivial character of $\Gal(H/H^+)$.
  Then one knows that \[ \ord_{s=0} L_p(\chi\omega, s) \ge r. \]  This follows for odd $p$ by work of Wiles \cite{wiles}; an alternate proof using the Eisenstein cocycle that works for all $p$ was given in \cite{cd} and \cite{spiess} using an argument of Spiess.
  In the papers \cite{ddp} (joint with Darmon and Pollack) and \cite{dkv} (joint with Kakde and Ventullo), we proved that
\[ L_p^{(r)}(\chi\omega,0) = (*) R_p^-(H) \]
where $(*)$ is a specific non-zero rational number.
This is a $p$-adic class number formula at $s=0$.  Therefore, $ L_p(\chi\omega, s)$ has a zero of order exactly $r$ at $s=0$ if and only if the Gross-Kuz'min conjecture is true.

\subsubsection{Representation theoretic considerations}
 
 Retaining the setting of the Gross--Kuz'min Conjecture, suppose now that $H$ contains a totally real field $F$ such that $H/F$ is Galois.  Let $G = \Gal(H/F)$.  For any representation $M$ of $G$ over $\C_p$, and character $\chi$ of an irreducible representation $V$, let $M^\chi$ denote the $\chi$-isotypic component of $M$ (i.e.\ the span of the subrepresentations of $M$ isomorphic to $V$). 
 
Then \[ U_p^- = \bigoplus_{\chi} U_p^\chi, \qquad X_p^- = \bigoplus_{\chi} X_p^\chi, \]
where the sums range over the characters $\chi$ of irreducible representations $V$ of $G$ on which $c$
 acts as $-1$.  The maps $\ell_p$ and $o_p$ also decompose as sums of maps 
 \[ \ell_{p}^\chi, o_p^\chi \colon U_p^\chi \longrightarrow X_p^\chi. \]
We define 
 \[   R_p^\chi(H) = \det(\ell_p^\chi \circ (o_p^\chi)^{-1}).\] We then have
 \begin{equation} \label{e:factor}
R_p^-(H) = \prod_\chi R_p^\chi(H). 
  \end{equation}
 Now, for $\chi$ as above, 
 \[ r_p^\chi := \dim_{\C_p} U_p^\chi = \dim_{\C_p} X_p^\chi = \sum_{\fp \mid p} \dim_{\C_p} V^{G_\fp}, \]
 where the sum ranges over the primes of $F$ above $p$, $G_{\fp} \subset G$ denotes the decomposition group of a prime of $H$ above $\fp$, and $V^{G_\fp}$ denotes the maximal subspace of $V$ invariant under $G_\fp$.
When $r_p^\chi = 1$, the regulator $ R_p^\chi(H)$ is a $\overline{\Q}$-linear combination of $p$-adic logarithms of algebraic numbers.  As pointed out by Gross in \cite{gross}*{Proposition 2.13}, the nonvanishing of $R_p^\chi(H)$ follows from the
theorem of Brumer--Baker (Theorem~\ref{t:brumer}) in this case.
 
 \begin{theorem} \label{t:rankone}
  If $r_p^\chi = 1$, then $R_p^\chi(H) \neq 0$.
 \end{theorem}
 
 There is a particular case when {\em every} $r_p^\chi \le 1$.  If $F$ contains only one prime above $p$ (for example $F = \Q$), and $G$ is abelian (so every $V$ has dimension 1), then clearly $r_p^\chi \le 1$.  Combining Theorem~\ref{t:rankone} with the factorization (\ref{e:factor}), we obtain:
 
 \begin{corollary}  Let $F$ be a totally real field with exactly one prime above $p$, and let $H$ be a CM abelian extension of $F$.  Then the Gross--Kuz'min conjecture holds for $H$.
 \end{corollary}
 
 A similar analysis holds for Leopoldt's conjecture, and we obtain:
 
  \begin{theorem} Leopoldt's conjecture holds for abelian extensions of $\Q$.
 \end{theorem}
  
 \subsection{A theorem of Roy}

 Damien Roy has proven a number of beautiful results in transcendence theory.  We prove one of these now.
 
 \begin{theorem}[Roy] \label{t:roy2}
  The Structural Rank Conjecture is equivalent to the special case of Schanuel's Conjecture that states that if $y_1, \dots, y_n \in \sL$ are $\Q$-linearly 
independent, then
\[ \trd_\Q \Q(y_1, \dots, y_n) = n. \]
 \end{theorem}
 
 Similarly, the $p$-adic Structural Rank Conjecture is equivalent to the $p$-adic version of the special case of Schanuel's conjecture, but we will content ourselves with the archimedean setting here.  Theorem~\ref{t:roy2} is proven in \cite{roy}.

\bigskip
One direction of Roy's Theorem is relatively elementary.

\begin{lemma}  The special case of Schanuel's Conjecture implies the Structural Rank Conjecture.
\end{lemma}

\begin{proof}
We assume the special case of Schanuel's conjecture.  We first consider a matrix $M$ with coefficients in $\sL$.
 Let $M = \sum M_i c_i$ with $M_i \in M_{m \times n}(\Q)$ and $c_i \in \sL$ linearly independent over $\Q$.  
 Write \[ M_x = \sum M_i x_i \in M_{m \times n}(\Q(x_1, \dots, x_n)) \] and let $r=\rank(M_x)$.  Let $J_x$ be an $r \times r$ submatrix of $M_x$ such that \[ \det(J_x) = P(x_1, \dots, x_n) \neq 0 \]
 in $\Q[x_1, \dots, x_n].$  The determinant of the corresponding submatrix of $M$ is equal to $P(c_1, \dots, c_n)$ and hence cannot vanish since the $c_i$ are algebraically independent, by the special case of Schanuel's conjecture.  Therefore $\rank(M) \ge r$.  Of course it is clear that $\rank(M) \le r$, so we get equality.
 \bigskip

Now assume $M$ has coefficients in $\sL + \Q$, but not in $\sL$.  There are 2 cases. 
\bigskip

{\bf Case 1: }  $1$ is not in the $\Q$-linear span of the coefficients of $M$.  The $\Q$-basis for this span can be taken to have the form $1+c_1, c_2, \dots, c_n$, where $c_i \in \sL$.  It is easy to check that the $c_i$ must be $\Q$-linearly independent, and hence by the special case of Schanuel's conjecture, they are algebraically independent.  The same is therefore true of $1+c_1, c_2, \dots, c_n$.  The previous proof then applies to this basis.

\bigskip

{\bf Case 2: }  $1$ is in the $\Q$-linear span of the coefficients of $M$.  We may take a $\Q$-basis of this span of the form $c_0 = 1, c_1, \dots, c_n$, where $c_i \in \sL$ for $i \ge 1$.  
We proceed as before. Write 
\[ M = \sum_{i=0}^n M_i c_i, \quad M_x = \sum_{i=0}^n M_i x_i. \]
Let $r = \rank(M_x)$ and $J_x$ an $r \times r$ submatrix of $M_x$ with \[ \det(J_x) = P(x_0, \dots, x_n) \neq 0. \]
The determinant of the corresponding submatrix $J$ of $M$ is $P(1, c_1, \dotsc, c_n)$. 
Since $P$ is a nonzero homogeneous polynomial of degree $r$, its specialization $P(1, x_1, \dots, x_n)$ is also nonzero, so 
$\det(J) = P(1, c_1, \dotsc, c_n) \neq 0$ by the algebraic independence of the $c_i$.  Therefore
 $\rank(M) \ge r$ as desired.
\end{proof}

  The main content of the converse is in the following lemma.

\begin{lemma} \label{l:roy}  Let $k$ be a commutative ring and let $P \in k[x_1, \dots, x_n]$. There exists a square matrix $N$ with coefficients in \[ k + k x_1 + \dots + k x_n \] such that $\det(N) = P$.
\end{lemma}

Let us for the moment take the lemma for granted and prove Roy's Theorem.

\begin{proof}[Proof of Theorem~\ref{t:roy2}]

Assume the Structural Rank Conjecture. Suppose $c_1, \dots, c_n \in \sL$ are linearly independent over $\Q$ and that 
$ P(c_1, \dots, c_n) = 0 $ for some nonzero $P \in \Q[x_1, \dots, x_n]$.
As in Lemma~\ref{l:roy}, let $N$ be a square matrix with coefficients in $\Q + \Q x_1 + \dots + \Q x_n$ 
such that $\det(N) = P$.

 Let $M$ be the matrix $N$ with $x_i$ replaced by $c_i$.  We then have $\det(M) =0$. 
Note that the matrix $M_x$ in the Structural Rank Conjecture is the homogenization of the matrix $N$, with
 coefficients in $\Q x_0 + \Q x_1 + \cdots + \Q x_n$.  We are using here that the $c_i$ are $\Q$-linearly independent from 1, since $e$ is transcendental.  The conjecture implies that $\det(M_x) = 0$, whence $\det(N) = 0$ by specializing $x_0=1$, a contradiction.
 \end{proof}
 
 It remains now to prove Lemma~\ref{l:roy}.  We first remark that this lemma is actually the starting point of an important avenue of research in theoretical computer science, where the lemma is usually attributed to Valiant.  There are well-known efficient algorithms for calculating the determinant of a matrix, so expressing a general polynomial as a determinant gives an algorithm for efficiently calculating values of a polynomial.  The minimal dimension of matrix necessary to express a given polynomial as a determinant is known as the {\em determinantal complexity} of the polynomial.  The study of the growth of determinantal complexity in families of polynomials is a topic with an extensive literature.
 
 We follow Roy's proof of Lemma~\ref{l:roy}.  We need to establish two sublemmas.
 
\begin{lemma} \label{l:ab} For a nonegative integer $d$, 
 let $P_d \subset k[x_1, \dots, x_n]$  denote the $k$-subspace of polynomials of total degree $\le d$.
Given $N \in M_{m \times m}(P_d)$ with $d \ge 1$, there exists an integer $s$ and matrices 
\[ A \in M_{m \times s}(P_{d-1}), \quad B \in M_{s \times m}(P_1) \] such that $N = AB$.
\end{lemma}

\begin{proof}
Let $N = (a_{i,j})$ with $a_{i,j} \in P_d$. We can write \[a_{i,j} = \sum_{\ell=1}^{n} c_{i,j,\ell} x_\ell + c_{i,j,n+1}\] with $c_{i,j,\ell} \in P_{d-1}$ for $1 \le \ell \le n+1$. Let
\[ c_{i,j} = (c_{i,j,\ell}) \in M_{1 \times (n+1)}(P_{d-1}), \quad x = \begin{pmatrix} x_1 \\ x_2 \\ \vdots \\ x_n \\ 1 \end{pmatrix}
 \in M_{(n+1) \times 1}(P_1). \]
Define
\[ A = (c_{i,j}) \in M_{m \times m(n+1)}(P_{d-1}), \quad B = x \otimes 1_{m \times m} \in M_{(n+1)m \times m}(P_1). \]
Then one calculates that $N=AB$.
\end{proof}

The matrices $A$ and $B$ in Lemma~\ref{l:ab} are not square, so we cannot recursively apply the lemma.  This is resolved by the following observation.

\begin{lemma} \label{l:square}
  Let $A \in M_{m \times s}, B \in M_{s \times m}$. Then \[ \det(AB) = \det\begin{pmatrix} I_s & B \\ -A & 0 \end{pmatrix}, \]
  where the matrix on the right is square of dimension $m+s$.
  \end{lemma}

\begin{proof}  We simply note that
\[ \begin{pmatrix} I_s & 0 \\ A & I_m \end{pmatrix} \begin{pmatrix} I_s & B \\ -A & 0 \end{pmatrix}\begin{pmatrix} I_s & -B \\ 0 & I_m \end{pmatrix} = \begin{pmatrix} I_s & 0 \\ 0 & AB \end{pmatrix} \]
and take determinants of both sides.
\end{proof}

We can now prove our main lemma.

\begin{proof}[Proof of Lemma~\ref{l:roy}]
 More generally, it now follows by induction on $d$ that for any matrix $N \in M_{m\times m}(P_d)$, there exists a matrix $N' \in M_{m' \times m'}(P_1)$ such that $\det(N) = \det(N')$. 

The base case $d=1$ is trivial. For  $d > 1$ we use Lemma~\ref{l:ab} to write $N=AB$ with $A \in M_{m \times s}(P_{d-1})$ and $B \in M_{s \times m}(P_1)$.   Lemma~\ref{l:square} then yields $\det(N) = \det(N')$ with $N' \in M_{(m+s) \times (m+s)}(P_{d-1})$.  The induction is now complete.  

The lemma is the case where we start with a $1\times 1$ matrix in $P_d$.
\end{proof}

 \section{The theorem of Waldschmidt and Masser} \label{s:wm}
 
To our knowledge, the strongest general unconditional result toward the Structural Rank Conjecture is Theorem~\ref{t:wm} of   Waldschmidt and Masser stated in the introduction \cite{w}.  For the sake of variety, we will prove the $p$-adic version of the conjecture in this section, though the proof of the archimedean version is essentially the same.
The statement of the $p$-adic version is exactly the same as the archimedean one, with $\sL$ replaced by $\sL_p$.

\begin{theorem}[Waldschmidt--Masser] \label{t:wm2} Let $m, n$ be positive integers and 
let $M \in M_{m \times n}(\sL_p)$.  Suppose that \[ rank(M)  < mn/(m+n). \]
Then there exist $P \in \GL_m(\Q)$ and $Q \in \GL_n(\Q)$ such that $PMQ = \mat{M_1}{0}{M_2}{M_3}$ 
where the 0 block has dimension $m' \times n' $ with $m'/m + n'/n > 1$.
\end{theorem}

\subsection{Applications}

Let us state some applications of the complex and $p$-adic Waldschmidt--Masser theorems.
The six exponentials theorem, which had been proven earlier in the 1960s, is a corollary of the Waldschmidt--Masser theorem.

\begin{proof}[Proof of Theorem~\ref{t:six}]  The case where $M=0$ is trivial. Therefore suppose $M \in M_{2 \times 3}(\sL)$ has rank 1. Since $1 < 6/5$, the 
Waldschmidt--Masser theorem implies that after a rational change of basis on the left and right, the matrix $M$ has the block matrix form \[ PMQ = \mat{M_1}{0}{M_2}{M_3} \] where the 0 block has dimension $1 \times 2$ or $2 \times 1$.
In the first case, our matrix has the form
\[ PMQ = \begin{pmatrix} * & 0 & 0 \\
* & * & *
\end{pmatrix}.
\]
Such a matrix has rank 1 only if it has the form
\[ 
PMQ = \begin{pmatrix} 0 & 0 & 0 \\
* & * & *
\end{pmatrix} \quad \text{ or } \quad 
PMQ = \begin{pmatrix} * & 0 & 0 \\
* & 0 & 0
\end{pmatrix}.
\]
In the first case, we see that $PM$ has the same shape, which says that the rows of $M$ are linearly dependent over $\Q$. In the second case we see that $MQ$ has the same shape, which says that the columns of $M$ are linearly dependent over $\Q$.  The case where the original block of 0's has dimension $2 \times 1$ is similar.
\end{proof}

In the case of a square matrix, the Waldschmidt--Masser theorem simplifies to the following. 
 
\begin{corollary}
Let $M \in M_{n \times n}(\sL)$ or $M_{n \times n}(\sL_p)$.  Suppose that $\rank(M) < n/2$.  Then there exist $P,Q \in \GL_n(\Q)$ such that
\[ PMQ = \begin{pmatrix}
M_1 & 0 \\
M_2 & M_3 
\end{pmatrix} \quad \text{(block matrix)}
\]
 where the $0$ block has dimension $m \times m'$ with $m + m' > n$.
\end{corollary}

\begin{corollary}
  The Leopoldt regulator matrix and the Gross--Kuz'min regulator matrix have rank at least half their expected ranks. 
  \end{corollary}

\begin{proof}  Let $r$ be the expected rank of the Leopoldt matrix.  Let \[ M' = (\log |\sigma_j(u_i)|)_{i,j=1, \dotsc, r}\] be
an $r \times r$ submatrix of the {\em archimedean} regulator with $\det(M') \neq 0$.  Let \[ M = (\log_p \sigma_j(u_i))_{i,j=1, \dotsc, r} \] be the corresponding submatrix of the Leopoldt matrix.  

If the rank of the Leopoldt matrix is less than $r/2$, the same is true for $M$. The Waldschmidt--Masser theorem then implies that there exist $P, Q \in \GL_r(\Q)$ such that $PMQ$ has an upper right 0 block with dimension $m \times m'$, where $m + m' > r$. 
 But then $PM'Q$ has this same property.  This implies that $\det(M') =0 $, a contradiction. 

The same proof works for Gross's regulator, using $\ord_p$ instead of $\log |\cdot |$. 
\end{proof}

\subsection{Auxiliary Polynomial} \label{s:aux}

As in the proof of Baker's theorem, the Waldschmidt--Masser theorem is proven by constructing, under the assumptions of the theorem, a suitable  auxiliary polynomial whose existence implies the conclusion of the theorem.  Waldschmidt's result is that the auxiliary polynomial exists, and Masser's theorem is that this polynomial gives the desired conclusion.  Let us describe this in greater detail.

We have $M=(a_{i,j})$ with $a_{i,j} = \log_p(x_{i,j}) \in \sL_p$.  Here $x_{i,j}  \in \overline{\Q}^*$.  After scaling $M$ if necessary, we may assume that $|x_{i,j} - 1|_p < 1$.
  For $i=1,\dots, m$, let 
\[ x_i = (x_{i,j})_{j=1,\dotsc, n} \in (\overline{\Q}^*)^n \subset (\C_p^*)^n. \]
Let $X = \langle x_i \rangle \subset (\overline{\Q}^*)^n$ be the subgroup generated by the $x_i$.   For each positive integer $N$, define
\[ X(N) = \left\{ \prod_{i=1}^{m} x_i^{a_i} \colon 0 \le a_i \le N \right\}. \]

For a polynomial $P$ in several variables, we write $\deg(P)$ for the total degree of $P$.

\begin{theorem}[Waldschmidt] \label{t:w}
Suppose $r=\rank(M) < mn/(m+n)$.  There exists $\epsilon > 0$ such that for all $N$ sufficiently large, there exists a nonzero $P \in \Z[t_1, \dots, t_n]$ such that $\deg(P) < N^{m/n - \epsilon}$ and $P(x) = 0$ for all $x \in X(N)$.
\end{theorem}

Waldschmidt's theorem is the ``transcendence'' part of  Theorem~\ref{t:wm}.   Masser's theorem, which is a purely algebro-geometric statement, takes the existence of an auxiliary polynomal $P$ as above and deduces the relations
necessary to give the desired result about the original matrix $M$. We will describe  the statement of Masser's theorem precisely in a moment, but first let us comment about the numerology concerning the auxiliary polynomial in the statement of Theorem~\ref{t:w}.  We can view the existence of a polynomial with prescribed zeroes as a system of linear equations in the coefficients of the polynomial.  Each zero gives one such linear equation.  If the $x_{i,j}$ are generic, the size of $X(N)$ is $(N+1)^m$. 
A polynomial of degree  $ < d$ has less than $d^n$ coefficients.  Therefore, if the $x_{i,j}$ are generic, we expect that we would require $d^n \ge (N+1)^m$ for a polynomial to exist, in particular $d > N^{m/n}$.  For this reason, the existence of the auxiliary polynomial $P$ in Theorem~\ref{t:w} does not hold for generic ${x_{i,j}}$.

Let us now state Masser's Theorem precisely.
Let $k$ be a field of characteristic 0, let $(x_{i,j}) \in M_{m\times n}(k^*)$. Define $X$ and $X(N)$ as above.
Define a pairing \begin{align*} \Z^m \times \Z^n &\longrightarrow k^* \\
\langle (a_i), (b_j) \rangle &= \prod_{i,j} x_{i,j}^{a_ib_j}.
\end{align*}

\begin{theorem}[Masser] \label{t:m}
 Let $N > 0$ and suppose there exists $P \in k[t_1, \dots, t_n]$ such that $\deg(P) < (N/n)^{m/n}$ and $P(x)=0$ for all $x \in X(N)$.  Then there exist subgroups $A \subset \Z^m$, $B \subset \Z^n$ of ranks $m', n'$, respectively, with $\langle A, B \rangle = 1$ and $m'/m + n'/n > 1$.
\end{theorem}

Theorems~\ref{t:w} and~\ref{t:m} combine to give Theorem~\ref{t:wm2}.  In the remainder of this section, we prove these two theorems.

\subsection{Waldschmidt's Theorem}

We will present two proofs of Waldschmidt's Theorem.

\subsubsection{Proof 1 of Waldschmidt's Theorem}
Our first proof is similar in spirit to Waldschmidt's original proof.
For simplicity, we will assume $x_{i,j} \in \Z$ and $x_{i,j} \equiv 1 \pmod{p}$.  Standard techniques (scaling by an integer to obtain algebraic integers, and taking norms to obtain integers) allow one to handle the general case, but we would like to avoid the extra notation required.  

Let $r$ denote the rank of the matrix $M \in M_{m \times n}(\Z_p)$.  After reordering columns if necessary, we can 
 assume that the last $n-r$ columns of $M$ are in the $\Z_p$-linear span of the first $r$ columns.  Then for each $i>r$, there exist $\lambda_{i,1}, \dots, \lambda_{i,r} \in \Z_p$ such that if $z=(z_1, \dotsc, z_n) \in X$, we have 
\begin{equation}
\label{e:zir}
 z_i = z_1^{\lambda_{i,1}} z_2^{\lambda_{i,2}} \cdots z_r^{\lambda_{i,r}}, \qquad i > r. 
\end{equation}
To make sense of the right hand side of this equality, note that
that for $\lambda \in \Z_p$, the function \begin{equation} \label{e:tl}
 t^\lambda = (1+ (t-1))^\lambda = \sum_{i=0}^{\infty} \binom{\lambda}{i} (t-1)^i 
 \end{equation} is a convergent power series in $t-1$.   Hence if $t \in 1 + p\Z_p$ then (\ref{e:tl}) converges in $\Z_p$.

Our goal is to find a polynomial $P \in  \Z[t_1, \dots, t_n]$ such that $P(z) = 0$ for $z \in X(N)$.
 Define $u_i = t_i - 1$ and consider the canonical map
\begin{align}
 \varphi\colon \Z[t_1, \dots, t_n] \longrightarrow & \ \Z_p[[u_1, \dots u_n]]/(t_i - t_1^{\lambda_{i,1}} \cdots t_r^{\lambda_{i,r}})_{i=r+1}^{n} \label{e:zq} \\
& \cong \Z_p[[u_1, \dots, u_r]]. \nonumber
\end{align}
The elements in the quotient in (\ref{e:zq}) are interpreted as power series in the $u_i$ via (\ref{e:tl}).

Fix a positive integer $c$. 
Define $\varphi_c$ to be the composition of $\varphi$ with the canonical reduction
\begin{equation} \label{e:pc}
 \Z_p[[u_1, \dots, u_r]] \longrightarrow (\Z/p^c\Z)[[u_1, \dots, u_r]]/(u_1^c, \dots, u_r^c). 
 \end{equation}
 If a polynomial $P \in  \Z[t_1, \dots, t_n]$ satisfies $\varphi_c(P) = 0$,  then $P(z)$ will be divisible by $p^c$ for any $z \in X$. Indeed, for $z \in X$,  $\varphi(P)(z) = P(z)$ is well-defined since the kernel of $\varphi$ vanishes on $X$.   Next, it is clear that $\varphi(P)(z) \pmod{p^c}$ depends only on the coefficients of $\varphi(P)$ modulo $p^c$.  Finally, we note that $z_i \equiv 1 \pmod{p} \Longrightarrow u_i \equiv 0 \pmod{p} \Longrightarrow u_i^c \equiv 0 \pmod{p^c}$.
 
 Now, the ring on the right in (\ref{e:pc}) is finite.  
The total number of monomials in $u_1, \dots, u_r$ modulo $(u_1^c, \dots, u_r^c)$ is 
$c^r$, so the total number of possible values of these coefficients mod $p^c$ is \[ (p^c)^{c^r} = p^{c^{r+1}}.\]
Therefore, by the Pigeonhole Principle, if we have a subset of $\Z[t_1, \dots, t_n]$ of size greater than $p^{c^{r+1}}$, then some two elements of the subset, say $P_1$ and $P_2$, will have equal image under $\varphi_c$, and the difference $P = P_1 - P_2$ will satisfy $P(z) \equiv 0 \pmod{p^c}$ for all $z \in X$.

We will take the subset of all polynomials with degree in each variable less than some constant $d$ with coefficients that are nonnegative integers less than $p^h$, for some constant $h$.  The size of this subset is $p^{hd^n}$, and hence the condition that we want is \begin{equation} \label{e:ph}
hd^n > c^{r+1}.
\end{equation}

Now, we also want to use the principle of ``discreteness of the integers'' discussed in the proof of Baker's Theorem to ensure that the condition $P(z) \equiv 0 \pmod{p^c}$ for $z \in X(N)$ implies that $P(z) = 0$.
For this, we need a crude bound on $|P(z)|$ (archimedean absolute value).  Suppose that $A$ is an upper bound on
$|x_{i,j}|$.  Then for each $z = (z_1, \dots, z_n) \in X(N)$ we have $|z_i| < A^{Nm}$.
Therefore each monomial in the evaluation of $P(z)$ has absolute value at most $p^h A^{Ndmn}$, and in total we obtain
\[ |P(z)| < d^n p^h A^{Ndmn}. \]
Therefore if 
\begin{equation} \label{e:dh}
d^n p^h A^{Ndmn} < p^c
\end{equation}
 then we indeed have the implication \[ P(z) \equiv 0 \pmod{p^c}\Longrightarrow P(z) = 0. \]

To prove the theorem, we set $d = \lfloor N^{m/n - \epsilon} /n \rfloor$, so that the constructed polynomial $P$ will have degree less than $N^{m/n - \epsilon}$ as required.   We search for parameters $h$ and $c$ such that both (\ref{e:ph}) and (\ref{e:dh}) hold.   Not surprisingly, these inequalities are pulling in the opposite direction---the first says that $h$ is large relative to $c$, and the second says that $h$ is small relative to $c$.
For $N$ large, the two inequalities will be satisfied if 
\[ h \cdot N^{m - n \epsilon} \gg c^{r+1}, \qquad c > k(\log N) +  h + k' N^{(m+n)/n - \epsilon} \]
for the appropriate constants $k, k'$.

If we set $c = N^{(m+n)/n + \epsilon}$ and $h = c N^{-\delta}$ for a small $\delta > 0$, 
then it is clear that the second inequality will hold for $N$ sufficiently large.
Plugging these parameters into the first inquality yields \[ m - \delta > \left(\frac{m}{n} + 1 + \epsilon\right)r.  \]
It is clear that we can choose small positive $\delta, \epsilon$ satisfying this inequality if 
\[ m >  \left(\frac{m}{n} + 1 \right)r, \qquad
\text{i.e., } \qquad r < \frac{mn}{m+n}. \]
This completes the proof.

\subsubsection{Proof 2 of Waldschmidt's Theorem}

For our second proof, we will return to the completely general case, i.e.\ we do not assume that $x_{i,j} \in \Z$, only that $x_{i,j} \in \overline{\Q}^*$.  Our motivation in giving the second proof is that it introduces an important topic in transcendence theory not discussed earlier, namely the theory of {\em interpolation determinants} pioneered by Michel Laurent. In \cite{laurent}, Laurent gave a new proof of the six exponentials theorem using his new theory.  The basic idea is that we will view the existence of the desired polynomial $P$ as the solution of a linear system of equations in the coefficients of the polynomial, and show that the associated determinant vanishes.

Again we will construct a polynomial $P$ such that the degree in each variable is  less than $d = \lfloor N^{m/n - \epsilon} /n \rfloor$ such that $P(z) = 0$ for all $z \in X(N)$.
Consider the  matrix whose rows are indexed by our desired zeroes $z \in X(N)$, and columns are indexed by the exponents 
\[ y \in \Z^n(d-1) = \{ (y_1, \dots, y_n) \colon 0 \le y_i \le d-1 \}. \]
 of the monomials of our desired polynomial:
\[ L = (z^y) \in M_{(N+1)^m \times d^n}(\C_p). \]
It suffices to show that $\rank(L) < N^{m-n\epsilon}/n^m < d^n$, as then any nonzero vector in its kernel will be the coefficients of our desired polynomial. Of course by making $\epsilon$ smaller, we can ignore the constant $n^m$, since $N^\epsilon \gg n^m$ for $N$ large.

We state without proof the following elementary interpretation of the rank of a matrix.
\begin{lemma}
  Let $k$ be a field and suppose that a matrix \[ (a_{i,j}) \in M_{m \times n}(k) \] has rank equal to $r$.  Then there exist vectors \[ \beta_1, \dots, \beta_m, \gamma_1, \dots, \gamma_n \in k^r \] such that $a_{i,j} = \langle \beta_i, \gamma_j \rangle$.
\end{lemma}

In our situation, we have $M = \log_p(x_{i,j}) \in M_{m \times n}(\C_p)$ with rank $r$.
We write \[ \log_p(x_{i,j}) = \langle \beta_i, \gamma_j \rangle \qquad \text{ for }  \beta_i, \gamma_j \in \C_p^r. \]
Without loss of generality, we can scale all the $\beta_i$, $\gamma_j$ to assume all their coordinates have absolute value $< p^{-1}$.  (This just scales the matrix $M$, which affects neither the assumptions nor conclusions of the theorem.) 

If $z = \prod_{i=1}^m {x_i^{\ell_i}}$ for $\ell \in \Z^m$,  then for $y \in \Z^n$ we have
\[ z^y = \exp{\left\langle \sum \beta_i \ell_i, \sum \gamma_j y_j \right\rangle}. \]

Next we will require a $p$-adic Schwarz' Lemma.  For a positive integer $d$ and real $R > 0$, define
\[ B_d(R) = \{ (z_1, \dots, z_d)\colon |z_i| \le R \text{ for all } i\} \subset \C_p^d. \]
For analytic $f \colon B_d(R) \longrightarrow \C_p$, define
\begin{equation} \label{e:max}
 |f|_{R} = \max_{z \in B_d(R)} |f(z)|. 
 \end{equation}

\begin{lemma} Suppose that $f \colon B_1(R) \longrightarrow \C_p$ is analytic and has a zero of order at least $n$ at $z=0$.  Then for any $0 < R' < R$, we have
\[ |f|_{R'} \le \left(\frac{R}{R'}\right)^{-n} |f|_R. \]
\end{lemma}
\begin{proof}
Let $g(z) = f(z)/z^n$.  This is analytic on $B_1(R)$ since $f$ has a zero of order at least $n$ at $z=0$. 
For any $z \in B_1(R')$, we have
\begin{align*}
 |f(z)| & \le (R')^n|g(z)|  \\
& \le (R')^n|g|_R \\
& = \left(\frac{R'}{R}\right)^n|f|_R. 
\end{align*}
The last equality uses the $p$-adic maximal modulus principle, which states that the maximum in (\ref{e:max}) is achieved on the boundary $|z|= R$.  See \cite{cherry}*{Theorem 1.4.1} or \cite{stansifer}*{Theorem 7}  for a proof of this analytic fact.
\end{proof}

Now we present Laurent's main theorem on interpolation determinants.

\begin{theorem}[Laurent]  Let $0 < R' < R$ and let $f_1, \dots, f_d$ be analytic functions 
\[ B_r(R)  \longrightarrow \C_p. \]  Let $z_1, \dots, z_d \in B_r(R')$.
Then $L = \det(f_j(z_i))$ satisfies
\[ |L| \le \left(\frac{R}{R'}\right)^{-\Theta_r(d)} \prod_{i=1}^{d} |f_i|_R,\]
where for $d$ sufficiently large relative to $r$,
\begin{equation} \label{e:trd}
 \Theta_r(d) > \frac{r}{6e}d^{(r+1)/r}. \end{equation}
\end{theorem}

\begin{proof}
Define $\Delta(z)=\det(f_j(z_i z))$, which is analytic on $|z| \le R/R'$. We will show that $\Delta(z)$ has a zero of order at least $\Theta_r(d)$ at $z=0$, for some combinatorial function $\Theta_r$ satisfying (\ref{e:trd}) that we will define in a moment.  The result then follows from Schwarz' Lemma:
\[ |L| = |\Delta(1)| \le \left(\frac{R}{R'} \right)^{-\Theta_r(d)} |\Delta|_{R/R'} \]
using the trivial upper bound
\[ |\Delta|_{R/R'} \le \prod_{i=1}^{d} |f_i|_R. \]
(Note that in the complex case, we would need a factor of $d!$ on the right, but in the nonarchimedean setting, this factor is not required because of the strong triangle inequality.)

Write each $f_i$ as a power series in the variables $u_1, \dots, u_r \in \C_p$.  By multilinearity of the determinant, it suffices to consider the case $f_j(u)= u^{v_j} = u_1^{v_{1j}}u_2^{v_{2j}} \cdots u_r^{v_{rj}}$ 
for nonnegative integers $v_{ij}$.  Then \[ \Delta(z) = z^{\sum_j || v_j ||} \det(z_i^{v_j}), \]
where $|| v_j || = v_{j1} + v_{j2} + \cdots + v_{jr}$.
If any two tuples $v_j$ are equal, this determinant vanishes and $\Delta(z)$ is identically 0.  
If not, then the order of vanishing is at least 
\begin{align*}
 \Theta_r(d) := & \text{ the minimum of } \sum_{j=1}^d || v_j || \text{ as } v_j \text{ ranges} \\
&  \text{ over all distinct tuples of
elements of } (\Z^{\ge 0})^r. 
\end{align*}
For example, $\Theta_1(d) = d(d-1)/2.$
For a proof of the combinatorial inequality (\ref{e:trd}) see \cite{lil}*{Lemma 4.3}.
\end{proof}

We can now apply Laurent's theorem to complete the proof of Waldschmidt's theorem.
We want to show that any square submatrix $L' = (z_i^y)$ of $L$ of dimension $d^n \approx N^{m-n\epsilon}$ has vanishing determinant, 
where \[ z_1, \dots, z_d \in X(N), \quad y \in \Z^n(d-1), \quad d = \lfloor N^{m/n-\epsilon} \rfloor . \]

As explained earlier, the entries of the matrix $L'$ can be written
$\exp( \langle \sum \beta_i \ell_i, \sum \gamma_i y_i \rangle)$
with $\ell \in \Z^m(N)$ corresponding to $z$.
For each $y$ we have the function 
\[ f_{y}(u_1, \dots, u_r) = \exp(\langle u, \sum \gamma_i y_i \rangle) \]
We apply Laurent's theorem on interpolation determinants with $R = 1$ and $R' = 1/p$.  We find 
\[ |L'| \le C^{-N^{(m-n\epsilon)(r+1)/r}} \]
where $C > 1$ is a constant.

Now we want to put a bound on the archimedean absolute value of $L'$.
Let \[ A = \max_{i,j} |x_{i,j}|_\infty. \]  Then $|z^y|_\infty \le A^{N \cdot N^{(m/n) - \epsilon} n}$. 
Therefore
\[ |L'|_\infty \le (N^{m-n\epsilon})! \cdot D^{N^{(m/n) + 1 + m - (n+1)\epsilon}}. \]
The factorial is dominated by the other term and can be ignored.  Scaling to obtain integrality just scales $D$. The same is true for taking norm from the field generated by the $x_{i,j}$ down to $\Q$ in order to obtain an element of $\Z$.  

Therefore we will have $L' =0$ if
\[ C^{N^{(m-n\epsilon)(r+1)/r}} > D^{N^{(m/n)+1+m-(n+1)\epsilon}}. \]
Of course, for this inequality to hold for large $N$, the precise values of $C$ and $D$ do not matter; all that matters is that we have the corresponding inequality of exponents.

It therefore suffices to have
\[ (m-n\epsilon) \frac{r+1}{r} > \frac{m}{n} + 1 + m - (n+1)\epsilon. \]
This simplifies to 
\[ \frac{1}{r} > \frac{m+n}{mn} + \epsilon\left(\frac{n-r}{r} \right). \]
There exists $\epsilon > 0$ satisfying this inequality if and only if
\[ r < \frac{mn}{m+n}. \]
This gives the desired vanishing of $\det(L')$ and completes the second proof of Waldschmidt's theorem.

\subsection{Masser's Theorem}

We conclude this section by proving Masser's Theorem, stated in Theorem~\ref{t:m} above.  This is a purely algebro-geometric statement that does not involve the logarithm or exponential functions.  In particular we work over an arbitrary field $k$ of characteristic 0.
Recall the notation established in \S\ref{s:aux}.
 We let the group $X \subset (k^*)^n$ act on the polynomial ring $R = k[t_1, \dots, t_n]$ by \[ z\cdot f = f(z_1 t_1, z_2 t_2, \dots, z_nt_n). \]
 Recall that the subgroup $X$ is generated by elements $x_1, \dots, x_m$.
If $a \in \Z^m$, we write $x^a = \prod_{i=1}^{m} x_i^{a_i} \in X$. 
For a prime ideal $\fp \subset R$, let 
\[ \Stab_X(\fp) = \left\{a \in \Z^m \colon x^a \cdot \fp = \fp\right\}. \]

Before delving into the proof, it is instructive to consider the simplest case, $n = m = 2$.
 We let $N > 0$ and suppose there exists $P \in k[t_1, t_2]$ such that $\deg(P) < N$ and $P(x)=0$ for all $x \in X(2N)$.
 We want to show that either:
 \begin{itemize}
 \item[(A)] there is a nonzero $a \in \Z^2$ such that $x^a = (1,1)$, (this corresponds to $m' = 1, n' = 2$) or  
 \item[(B)] there exists a nonzero $b \in \Z^2$ such that $z^b = z_1^{b_1}z_2^{b_2} = 1$ for all $z \in X$ (this corresponds to $m'=2, n' = 1$).
 \end{itemize}
 
We factor $P$ into a product $\prod P_i$ of irreducibles.  We can assume that none of the $P_i$ are monomials, since monomials have no zeroes in $(k^*)^2$.
We will first show that if any $P_i$ satisfies $\rank(\Stab_X((P_i))) = 2$, then we are in the second case above.  This follows from Lemma~\ref{l:height} below, but it is relatively easy to see in this case explicitly. Indeed, if $t_1^{a_1} t_2^{a_2}$ is a monomial occuring in $P_i$, then the  equation $z P_i = \lambda P_i $ for $z \in X$ and $\lambda \in k^*$ yields
\[  z_1^{a_1}z_2^{a_2} = \lambda. \]
Letting $t_1^{a_1'} t_2^{a_2'}$ be some other monomial occuring in $P_i$ (recall we may assume that $P_i$ is not a monomial) we get a similar equation; dividing these two cancels $\lambda$  so we obtain
\[ z_1^{b_1}z_2^{b_2} = 1, \]
where  $b_i = a_i - a_i'$ for $i=1,2$ are not both zero.  If $\rank(\Stab_X((P_i))) = 2$ then this holds for all $z$ in a finite index subgroup of $X$, so replacing $(b_1,b_2)$ by an appropriate multiple, we are in case (B).

Therefore, we are left to consider the case where  each irreducible factor $P_i$ of $P$ satisfies $\rank(\Stab_X((P_i))) \le 1.$
In this case, we will show that there is a polynomial of the form \[ Q = \sum_{i=1}^k a_i (z_i \cdot P), \] where $a_i \in \Z$ and $z_i \in X(N)$, such that $P$ and $Q$ are relatively prime.
Let us first explain why this completes the proof.  Since $P$ vanishes on $X(2N)$, each polynomial $z \cdot P$ with $z \in X(N)$ vanishes on $X(N)$, hence the polynomial $Q$ vanishes on $X(N)$. Therefore both $P$ and $Q$ vanish on $X(N)$.  The set $X(N)$ has size $(N+1)^2$ {\em unless} we are in case (A) above.  But $\deg Q \le \deg P < N$, and the polynomials are coprime, so we would obtain a contradiction to Bezout's theorem if these polynomials had $(N+1)^2$ common zeroes.  We must therefore be in case (A).

 To see the existence of the polynomial $Q$, we first show that for each irreducible polynomial $P_i$, there exists $z_i$ 
 such that $z_i^{-1} \cdot P_i$ does not divide $P$, equivalently, $P_i$ does not divide $z_i \cdot P$.  This is established by counting.  Since 
 $\rank_X((P_i)) \le 1$, there are at least $N+1$ distinct ideals among the set $(z^{-1} \cdot P_i)$ as $z$ ranges over $X(N)$.  See Lemma~\ref{l:images} below for a proof.  But $P$ has degree  less than $N$, which is a bound on the number of irreducible factors, so some $z^{-1}\cdot P_i$ must not be a factor of $P$.  With these $z_i$ in hand, the existence of the linear combination $Q$ is an easy inductive argument using the Pigeonhole Principle; see Lemma~\ref{l:lincomb} below.
 
We now return to the general case.  Recall that the {\em height} $\height(\fp)$ of a prime ideal $\fp$ is the largest integer $r$ such that there exists a chain of distinct prime ideals
\[ \fp_0 \subset \fp_1 \subset \cdots \subset \fp_r = \fp. \]

\begin{lemma} \label{l:height}
Let $\fm = (t_1 -1, \dots, t_n - 1)$.  Let $\fp \subset \fm$ be a prime of height $n'$ and let $A = \Stab_X(\fp)$.  There exists a subgroup $B \subset \Z^n$ of rank $\ge n'$ such that 
$\langle A, B \rangle_X = 1$.
\end{lemma}

\begin{proof}
Let $B = \{y \in \Z^n \colon \langle A, y \rangle_X = 1 \}.$  Choose $Z \subset \Z^n$ such that \[ \Q^n = \Q B \oplus \Q Z.\] 
 We want to show that $s := \rank(Z) \le n - n'$.  Let $z_1, \dots, z_s$ be a basis for $Z$.  Write $z_i = (z_{i,1}, \dots, z_{i,n})$. 

For $i=1, \dots, s$, let $u_i = \prod_{j=1}^n t_j^{z_{i,j}} \in R' = k[t_1^{\pm 1}, \dots, t_n^{\pm 1}]$.  Since \[ \trd_k \Frac(R'/\fp R') = n - n', \] if $s > n-n'$ then there exists a nonzero polynomial $Q$ with coefficients in $k$ such that $Q(u_1, \dots, u_s) \in \fp R'$.
Suppose this is the case, and write $Q(u_1, \dots, u_s)$ as polynomial $Q'(t_1, \dots, t_n) \in \fp R'$.

For any $a \in A$, we have $x^a \cdot Q' \in \fp R'$, so 
\begin{align*}
 Q'(x^a t) \in \fp R' \subset \fm R' & \Longrightarrow Q'(x^a) = 0 \\
& \Longrightarrow Q(\langle a, z_1 \rangle_X, \dots, \langle a, z_s \rangle_X) = 0. 
\end{align*}
Fix $a$ and apply this with $a$ replaced by $da$, as $d=0, 1, \dots$.  Using the Vandermonde trick from Baker's theorem, we find that some $\Z$-linear combination of the $z_i$ is orthogonal to $a$.  More precisely, we have
$ \langle a, w \rangle_X = 1 $
for some
\[ w \in S = \left\{ \sum_{i=1}^s w_i z_i \neq 0, \ |w_i| \le \deg(Q) \right\}. \]

 Therefore
\[ A = \bigcup_{w \in S} w^\perp. \]
But $A$ is a finitely generated free abelian group and cannot be written as a finite union of proper subgroups. Therefore there exists $w \in S$ such that $\langle A, w \rangle = 1$.  But then $w \in B$, contradicting $w \in Z$.
 Therefore $s \le n-n'$ as desired. 
\end{proof}

Given Lemma~\ref{l:height}, our task now is to show the existence of a prime ideal $\fp$ with height $n'$ such that $\rank(\Stab_X(\fp)) = m'$ where $m'/m + n'/n > 1$.  This is provided by the following theorem.

\begin{theorem} \label{t:Pexists}
 Let $N > 0$ and suppose there exists \[ P \in k[t_1, \dots, t_n] \] such that $\deg(P) < (N/n)^{m/n}$ and $P(x) = 0$ for all $x \in X(N)$.  Then there exists a prime ideal $\fp \subset \fm$ of height $n'$ such that 
\[ \rank(\Stab_X(\fp)) = m' \quad \text{ where } \quad m'/m + n'/n > 1. \]
\end{theorem}

Lemma~\ref{l:height} and Theorem~\ref{t:Pexists} combine to give Theorem~\ref{t:m}.
 We will prove the contrapositive of Theorem~\ref{t:Pexists}.
For each $1 \le n' \le n$, let $m' = m'_{n'}$ be the maximal rank of $\Stab_X(\fp)$ as $\fp$ ranges over the primes contained in $\fm$ with height equal to $n'$.  If any $m' = m$ then $m'/m + n'/n = 1 + n'/n > 1$, so we are done.  Therefore assume that every $m' < m$ and define
 \[ \eta_{n'} = \frac{n'}{m - m'}. \]
Note that 
\begin{equation} \label{e:nprime}
\eta_{n'} > n/m \Longleftrightarrow m'/m + n'/n > 1.
\end{equation}
Theorem~\ref{t:Pexists} will arise as a corollary of the following statement.
\begin{theorem} \label{t:deta}
 Let $f \in R$ have degree $D$ and let 
\[ N = D^{\eta_1} + D^{\eta_2} + \cdots + D^{\eta_n}. \]
There exists $z \in X(N)$ such that $f(z) \neq 0$. 
\end{theorem}
Theorem~\ref{t:deta} implies Theorem~\ref{t:Pexists}.  Indeed, if each $\eta_{n'}$ for $1 \le n' \le n$ satisfies 
$\eta_{n'} \le n/m$, then Theorem~\ref{t:deta} implies that there exists $z \in X(n \deg(P)^{n/m})$ such that $P(z) \neq 0$.  But by assumption $n \deg(P)^{n/m} < N$, yielding a contradiction to the assumption $P(z) = 0$ for all $z \in X(N)$.  Therefore some $\eta_{n'}$ is larger than $n/m$, giving the desired result by (\ref{e:nprime}).

\bigskip

The proof of Theorem~\ref{t:deta} requires significant commutative algebra.  We first establish some notation.
Let 
 \[ \fM = \bigcup_{z \in X(N)} z\cdot \fm, \quad S_{\fM} = R - \fM. \]
The set $S_{\fM}$ is multiplicatively closed.  For an ideal $\fa \subset R$, define 
\[ \fa^* = (S_\fM^{-1} \fa) \cap R \supset \fa. \]
Note that for a prime ideal $\fp \subset R$, we have $\fp^* = \fp \Leftrightarrow \fp \subset z \cdot \fm$ for some $z \in X$, and $\fp^* = R$ otherwise.  Indeed, if $\fp^* \neq \fp$, then there exists $t/s \in (S_{\fM}^{-1} \fp) \cap R$ such that $t/s \not \in \fp$.  Write $t/s = g \in R$, with $g \not \in \fp$.  Since $t = gs \in \fp$ and $\fp$ is prime, this implies that $s \in \fp$.  Since $s \in S_{\fM}$ we conclude that $\fp \not \subset \fM$, and hence $\fp \not\subset z \cdot \fm$ for any $z \in X(N)$.  Furthermore in this case we have $s/s =1 \in \fp^*$, so $\fp^* = R$.  Now, all of these steps are clearly reversible, except possibly the implication  $\fp \not \subset \fM \Longrightarrow \fp \not\subset z \cdot \fm$ for all $z \in X(N)$.
The converse of this statement reads $\fp \subset \fM \Longrightarrow \fp \subset z \cdot \fm$ for some $z \in X(N)$.  This is precisely the prime avoidance lemma.  This completes the proof of our claim about $\fp^*$.

We next recall  some definitions from commutative algebra.  An {\em associated prime} of an ideal $\fa \subset R$ is a prime ideal $\fp$ such that there exists an $R$-module injection $R/\fp \hookrightarrow R/\fa$.  (The associated primes play the role of the irreducible factors in our simplified proof for $n=m=2.$)
An ideal $\fa \subset R$ is called {\em unmixed of height $r$} if all its associated prime ideals have height $r$.

Next we recall the definitions of {\em dimension} and {\em degree} of an ideal of $R$ and some of the basic properties of these functions. Let $R_0 = k[t_0, \dots, t_n]$.  For $f \in R$, let $f_0 \in R_0$ denote the homogenization of $f$, defined by padding each monomial of $f$ with the correct power of $t_0$ to obtain a homogeneous polynomial of degree $\deg(f)$. For an ideal $\fa \subset R$, let $\fa_0$ denote the homogeneous ideal generated by $f_0$ for $f \in \fa$. Then $R_0/\fa_0$ is a graded $R_0$-module.

There is a polynomial
\[ H_{\fa}(t) = a_d t^d + \cdots + a_0 \in \Q[x], \]
called the {\em Hilbert polynomial} of $\fa$, such that
\[ H_{\fa}(i) = \dim_k(i\text{th graded piece of } R_0/\fa_0) \]
for $i$ sufficiently large.
We define the dimension and degree of $\fa$, respectively by
\[ d(\fa) = d, \qquad \ell(\fa) := \tilde{\ell}(R_0/\fa_0) := a_d \cdot d!. \]
These are both integers.  They satisfy the following properties:
\begin{itemize}
\item $\ell((f))$ is the degree of $f$ in the usual sense.
\item If $\fa \subset \fb$ and $\height(\fa) = \height(\fb)$, then $\ell(\fa) \ge \ell(\fb)$.
\item If $\fa$ and $\fb$ are unmixed of height $r$, then so is $\fa \cap \fb$, and
\[ \ell( \fa \cap \fb)  \le \ell(\fa) + \ell(\fb). \]
Note that from this, it follows that if $\fa$ is unmixed then the number of associated primes of $\fa$ is $\le \fl(\fa)$.  To see this we note that there is a {\em primary decomposition} $\fa = \cap_{i=1}^r \fq_i$, where $\{\sqrt{\fq_i}\}$ is the set of associated primes.
\end{itemize}

We can now begin the proof of Theorem~\ref{t:deta}.
Let $f \in R$ have degree $D$ and let \[ N_r = D^{\eta_1} + \cdots + D^{\eta_{r-1}}, \qquad \text{ for } 1\le r \le  n+1. \]  We will inductively construct $f_r$, a $\Z$-linear combination of elements in $X(N_r) \cdot f$, such that $\fa_r = (f_1, \dots, f_r)$ satisfies the following: either $\fa_r^* = R$ or $\fa_r^*$ is unmixed of height $r$ and degree at most $D^r$. 

\bigskip

This will give the theorem: for $r=n+1$, $\fa_r^*$ cannot have height $n+1$, so $\fa_r^* = R$, which implies $\fa_r \not\subset \fM$.  In particular $f_i \not \in \fm$ for some $i$, so if $f_i = \sum d_j(z_j \cdot f)$ with $d_j \in \Z$ and $z_j \in X(N_{n+1})$ then $f(z_j) \neq 0$ for some $z_j$ as desired.

\bigskip

{\bf Base Case:} Take $f_1 = f$, $\fa_1 = (f)$.  Then $\fa_1^* = (f^*)$, where $f^*$ is the quotient of $f$ by any irreducible factors not lying in $\fM$.  If $f^* \neq 1$, then $(f^*)$  is unmixed of height 1 by Krull's principal ideal theorem, and has degree $\le D = \deg(f)$. 

\bigskip

{\bf Inductive Step: } Suppose $r \ge 2$ and that we have constructed $f_1, \dots, f_{r-1}$.  If $\fa_{r-1}^* = R$, then we can take $f_r = f$.  We have $\fa_r^* = R$, and we are done.  Therefore we suppose that $\fa_{r-1}^*$ is unmixed of height $r-1$ and degree at most $D^{r-1}$.  The construction of $f_r$ is slightly elaborate in this case, so let us outline the steps.

\begin{enumerate}
\item For any associated prime $\fp$ of $\fa_{r-1}^*$, show by counting that there exists $a \in \Z^m(D^{\eta_{r-1}})$ such that $x^{-a} \fp$ is not associated to $\fa_{r-1}^*$, i.e. that $\fp$ is not associated to $x^a\fa_{r-1}^*$.
\item Show that this implies there exists $1 \le i \le r-1$ such that $x^a f_i \not\in \fp$.
\item Show that this implies there exists a $\Z$-linear combination $f_r$ of these $x^a f_i$ that does not lie in any $\fp$ associated to 
$\fa_{r-1}^*$.
\item Letting $\fa_r = (\fa_{r-1}, f_r)$, show that $\fa_r^* = R$ or $\fa_r^*$ is unmixed of height $r$.
\end{enumerate}

It is perhaps worth pointing out here that the fourth point above is precisely the reason that associated primes appear in this proof---the key fact is that if an element $f_r$ does not lie in any prime associated to $\fa_{r-1}^*$, then the height of $\fa_r^* = (\fa_{r-1}, f_r)^*$ goes up by one (or $\fa_r^* =  R$). 
Let us now carry out the 4 steps above.

\bigskip

(1) Let $\fp$ be associated to $\fa_{r-1}^*$.
 Then $\fp \subset \fM$, so $\fp \subset z \cdot \fm$ for some $z \in X$, so $z^{-1} \cdot \fp \subset \fm$.  By definition, 
$\rank(\Stab_X(z^{-1} \cdot \fp)) \le m'_{r-1}$, whence $\rank(\Stab_X(\fp)) \le m'_{r-1}$.

\begin{lemma} \label{l:images}  Let $T$ be a positive integer. Let $\Z^m(T)$ denote the set of tuples $(a_1, \dots, a_m) \in \Z^m$ with $0 \le a_i \le T$ for each $i$.
If $H \subset \Z^m$ is a subgroup of rank $h$, then the image of $\Z^m(T)$ in $\Z^m/H$ has size at least $(T+1)^{m - h}$.  \end{lemma}

Before proving the lemma, we first note that it implies that 
 the image of $\Z^m(D^{\eta_{r-1}})$ in $\Z^m/\Stab_X(\fp)$ has size at least 
\[ (\lfloor D^{\eta_{r-1}} \rfloor +1)^{m - m'_{r-1}} > (D^{\eta_{r-1}})^{m - m'_{r-1}} = D^{r-1}. \]
Now, the number of primes associated to $\fa_{r-1}^*$ is at most its degree $\ell(\fa_{r-1}^*) \le D^{r-1}$.
Therefore, there exists $a \in \Z^m(D^{\eta_{r-1}})$ such that $x^{-a}\fp$ is not an associated prime of $\fa_{r-1}^*$. Equivalently,
 $\fp$ is not an associated prime of $x^a \fa_{r-1}^*$.  This completes the first step.
 
 \begin{proof}[Proof of Lemma~\ref{l:images}]
Choose $m-h$ elements of the canonical basis of $\Z^m$ that generate a subgroup $B$ such that $H \cap B = \{0\}$.  Then the canonical map from $\Z^m$ to $\Z^m/H$ is injective when restricted to $B$.  The result follows since $B \cap \Z^m(T)$ contains exactly $(T+1)^{m-h}$ elements.
\end{proof}

 \bigskip
 
(2)  We move on to the second step.  Since $\fp$ and $x^a \fa_{r-1}^*$ are unmixed of the same height $r-1$, but $\fp$ is not associated to $x^a\fa_{r-1}^*$, it follows that 
$x^a \fa_{r-1}^* \not\subset \fp$. 
This implies $x^a \fa_{r-1} \not\subset \fp$ since $\fp^* = \fp$.  Since \[ \fa_{r-1} = (f_1, \dots, f_{r-1}),\] this implies there exists $1 \le i \le r-1$ such that $x^a f_i \not \in \fp$. This completes the second step.

\bigskip

(3) Step 3 follows from a general lemma.

\begin{lemma} \label{l:lincomb}
 Let $\fp_1, \dots, \fp_s$ be prime ideals of $R$ and let \[ f_1, \dots, f_s \in R \] such that $f_i \not\in \fp_i$.  Then there exists a $\Z$-linear combination of the $f_i$ that does not lie in any $\fp_i$. 
\end{lemma}

\begin{proof}  Induction on $s$.  In the base case $s=1$, there is nothing to prove.  For larger $s$, suppose that $g$ is a $\Z$-linear combination of $f_1, \dots, f_{s-1}$ that does not lie in $\fp_1, \dots, \fp_{s-1}$.  If $g \not \in \fp_{s}$, then we can simply take $g$ and we are done.  So suppose $g \in \fp_s$.  

Consider all linear combinations $f_s + ag$ with $a \in \Z$.  
For each $a$, consider the set $S_a \subset \{ \fp_1, \dots, \fp_{s-1} \}$ consisting of the $\fp_i$ such that $f_s + ag \in \fp_i$.
There are $2^{s-1}$ possible subsets $S_a$. By the Pigeonhole Principle, if we take $a=0, \dots, 2^{s-1}$, then there must exist distinct $a, a'$ such that $S_a = S_{a'}$.  But if $f_s + ag, f_s + a'g \in \fp_i$ for $1 \le i \le s-1$, then $(a-a')g \in \fp_i$, whence $g \in \fp_i$ (since $k$ has characteristic 0), a contradiction.   Therefore $f_s + ag \not \in \fp_i$ for $i \le s-1$. 

Also, $f_s \not\in \fp_s$ but $g \in \fp_s$ implies $f_s + ag \not\in \fp_s$.  Therefore $f_s + ag$ is the desired linear combination.
\end{proof}

We can now complete step 3: We conclude that there is a $\Z$-linear combination $f_r$ of the $x^a f_i$ (where $1 \le i \le r-1$ and $a \in \Z^m(D^{\eta_{r-1}})$) such that $f_r$ does not lie in any associated prime of $\fa_{r-1}^*$.

\bigskip

(4) Step 4 will follow from the following lemma.

\begin{lemma} \label{l:step4}
 Let $\fa \subset R$ be unmixed of height $r-1$ and suppose $f \in R$ is not contained in any of the primes associated to $\fa$.   Let $\fb =\fa + (f)$.  Then either $\fb = R$ or $\fb$ has height $r$.  In the latter case, $\ell(\fb) \le \ell(\fa) \cdot \deg f$. 
\end{lemma}

\begin{proof}
 Let $\fa = \fq_1 \cap \dots \cap \fq_m$ be a minimal primary decomposition and let $\fp_i$ be the radical of $\fq_i$. 
If $\fp_i + (f) = R$ for all $i$, then for each $i$ there exists an element of the form $(1 - gf) \in \fp_i$, hence an element of the form $(1 - gf)^j \in \fq_i$.  The product of these lies in $\fa$.  This product is congruent to 1 modulo $f$, so $1 \in \fb = (\fa, f)$.  Therefore, assume that there exists some $\fp = \fp_i$ such that $\fp + (f) \neq R$.

By Krull's principal ideal theorem, the image $\overline{\fb}$ of $\fb$ in $R/\fp$ has height 1.
The inverse image of any associated prime of $\overline{\fb} \subset R/\fp$ in $R$ is a prime of height $(r-1)+ 1 = r$.  Therefore the height of $\fb$ is at most $r$, and since $\fb \supset \fa$, the height is at least $r-1$. 

But if the height of $\fb$ is $r-1$, then it has some associated prime $\fp'$ of height $r-1$.  But $\fp' \supset \fb \supset \fa$.  As $\fa$ is unmixed of height $r-1$, this implies that $\fp'$ is an associated prime of $\fa$.  But $f \in \fp'$ and we assumed $f$ was not contained in any associated primes of $\fa$.  This is a contradiction, so we must have that the height of $\fb$ is $r$.

To conclude, we note that $(\fa + (f))_0 \supset \fa_0 + (f)_0$, hence 
\begin{align*} \ell(\fb) = \tilde{\ell}(R_0/\fb_0) & \le \tilde{\ell}(R_0/(\fa_0 + (f)_0)) \\
&= \tilde{\ell}(R_0/\fa_0) \cdot \deg(f) \\
& = \ell(\fa) \cdot \deg(f).
\end{align*} 
The second to last equality requires explanation.  
Firstly, $f_0$ is not contained in any of the
associated primes of $\fa_0$  since $f$ is not contained in any of the
associated primes of $\fa$.
  This implies that multiplication by $f_0$ is injective on $R_0/\fa_0$.  This multiplication map has degree equal to $\deg(f)$ and cokernel equal to $R_0/(\fa_0 + (f_0))$, whence
\[ H_{\fa_0 + (f_0)}(t+\deg(f)) = H_{\fa_0}(t + \deg(f)) - H_{\fa_0}(t).\]
This yields $ \tilde{\ell}(R_0/(\fa_0 + (f)_0)) =  \tilde{\ell}(R_0/(\fa_0) \cdot \deg(f)$ as desired.
\end{proof}

We can now complete step 4.  We have $\fa_r = \fa_{r-1} + (f_r)$.  Let $\fb = \fa_{r-1}^* + (f_r)$.  Then $\fa_r^* \supset \fb$, and Lemma~\ref{l:step4} implies that either $\fb = R$ or $\fb$ has height $r$ and $\ell(\fb) \le D^{r-1} \cdot D = D^r$.  

If $\fb=R$ then of course $\fa_r^*=R$, so assume the latter case holds.  Let $\fp$ be an associated prime of $\fa_r^*$.  Then $\height(\fp) \ge \height(\fa_r^*) \ge r$.  We want to show equality.  We know $\fp \subset \fm'$ where $\fm' = z \fm$ for some $z \in X(N)$.  We can work in the localization $R_{\fm'}$, which is a regular local ring. The ideal $\fa_rR_{\fm'}$ is generated by $r$ elements so Krull's height theorem implies it has height at most $r$, hence it has height exactly $r$.  Therefore it is unmixed of height $r$ and hence the same is true of the associated prime $\fp$. 

Finally, $\fa_r^* \supset \fb$ and both are unmixed of height $r$ so $\ell(\fa_r^*) \le \ell(\fb) \le D^r$.  This completes the proof of step 4, and of Theorem~\ref{t:deta}.

\section{The Matrix Coefficient Conjecture}

Both the assumption and the conclusion of the Waldschmidt--Masser Theorem are quite strong.  For instance, in the case of a square matrix of dimension $n$ with coefficients in $\sL$ or $\sL_p$, one assumes that the rank of the matrix is less than $n/2$ and one concludes that after a rational change of basis on both sides one can arrange a large block of zeroes; precisely, a block of dimension $m' \times n'$ where $m' + n' > n$.

We would like a statement that is more sensitive, and gives a ``rational'' condition whenever the rank is not full.  To this end, we have formulated with Mahesh Kakde the following conjecture.  We call it the {\bf Matrix Coefficient Conjecture} because it states that a square matrix with coefficients in $\sL$ that is singular can be made to have at least 1 zero after a rational change of basis on the left and right.

\begin{conjecture}[D--Kakde] \label{c:mcc}
Let $M$ be a square matrix of dimension $n$ with coefficients in $\sL$ or $\sL_p$.  If $\det(M)=0$, then there exist nonzero vectors $w, v \in \Q^n$ such that $w^t M v = 0$.
\end{conjecture}

Despite its simplicity, Conjecture~\ref{c:mcc} remains quite deep: in the case $n=2$, it is easily seen to be equivalent to the Four Exponentials Conjecture.  We have proven the following about the Matrix Coefficient Conjecture:
\begin{itemize}
\item Conjecture~\ref{c:mcc} is implied by the Structural Rank Conjecture.
\item The version of Conjecture~\ref{c:mcc} over $\sL_p$ implies both Leopoldt's conjecture and the Gross--Kuz'min conjecture.
\end{itemize}

We have also developed a strategy to study Conjecture~\ref{c:mcc} using auxiliary polynomials, but unfortunately the construction of the necessary polynomials remains a mystery.  Our hope is that Conjecture~\ref{c:mcc} may be more tractable than the Structural Rank Conjecture.  We will explore Conjecture~\ref{c:mcc} further in forthcoming work.

\end{document}